%% file: manuscript_main_unmarked.tex
\documentclass[a4paper,fleqn]{cas-sc}

\usepackage[numbers]{natbib}

\usepackage{graphicx}
\usepackage{amsmath,amssymb,amsthm,amsfonts}
\DeclareMathAlphabet{\mathfrak}{U}{euf}{m}{n}
\usepackage[english]{babel}
\usepackage[normalem]{ulem}
\usepackage{mathtools}
\usepackage{algorithmicx}
\usepackage{algorithm}
\usepackage[noend]{algpseudocode}
\usepackage{color,xcolor}
\usepackage{hyperref}

\usepackage{subfig}

\usepackage{float}

\usepackage{booktabs}
\usepackage{siunitx}

\newtheorem{theorem}{Theorem}[section]
\newtheorem{remark}{Remark}[section]
\newtheorem{lemma}{Lemma}[section]

\newcommand{\avg}[1]{\overline{#1}}
\newcommand{\jump}[1]{\ensuremath{[\![ #1 ]\!]}}

\newcommand{\M}{\mathcal{P}}

\newcommand{\hv}{\hat{v}}
\newcommand{\hu}{\hat{u}}
\newcommand{\hq}{\hat{q}}
\newcommand{\hh}{\hat{h}}
\newcommand{\hb}{\hat{b}}
\newcommand{\hw}{\hat{w}}

\newcommand{\hF}{\hat{F}}

\newcommand{\hbe}{\hat{\beta}}
\newcommand{\hal}{\hat{\alpha}}

\newcommand{\kk}{\mathfrak{k}}

\newcommand\revOne[1]{\textcolor{black}{#1}}
\newcommand\revTwo[1]{\textcolor{black}{#1}}

\begin{document}
\let\WriteBookmarks\relax
\def\floatpagepagefraction{1}
\def\textpagefraction{.001}

% Short title
\shorttitle{High-order structure preserving method for stochastic Galerkin SWE - unification and 2D extension}    

% Short author
\shortauthors{\"{O}ffner et al.}

% Main title of the paper
\title [mode = title]{A high-order, structure preserving scheme for the stochastic Galerkin shallow water equations - unification and two-dimensional extension}
%comparison and extension} 

\author[1]{Philipp \"Offner}[orcid=0000-0002-1367-1917]

% Email id of the first author
\ead{mail@philippoeffner.de}

% Credit authorship
\credit{Conceptualization; Formal analysis; Methodology; Investigation; Writing - original draft}

\author[2]{Per Pettersson}[orcid=0009-0005-4259-8092]

% Email id of the second author
\ead{per.pettersson@norceresearch.no}

% Credit authorship
\credit{Conceptualization; Formal analysis; Methodology; Investigation; Writing - original draft}

\author[3]{Andrew R. Winters}[orcid=0000-0002-5902-1522]

% Corresponding author indication
\cormark[3]

% Email id of the third author
\ead{andrew.winters@liu.se>
}

% Credit authorship
\credit{Investigation; Visualization; Software; Writing - original draft}

% Address/affiliation
\affiliation[1]{organization={Clausthal University of Technology, Institute of Mathematics},
            city={Clausthal-Zellerfeld},
            postcode={38678}, 
            country={Germany}}

\affiliation[2]{organization={NORCE Norwegian Research Centre},
            city={Bergen},
            postcode={5008}, 
            country={Norway}}

\affiliation[3]{organization={Link\"{o}ping University, Department of Mathematics, Division of Applied Mathematics},
            city={Link\"{o}ping},
            postcode={581 83}, 
            country={Sweden}}

% Corresponding author text
\cortext[3]{Corresponding author}

%%%%%%
%%%%%%

\begin{abstract}
Recently, two independent research efforts have been made to study the stochastic Galerkin formulation of the shallow water equations. %In particular, 
Bender and \"Offner developed entropy-conservative discontinuous Galerkin (DG) methods to solve the stochastic shallow water equations \revTwo{in a stochastic Galerkin framework} using Roe variable transformation, while Dai, Epshteyn and collaborators proposed second-order, energy-stable and well-balanced schemes for the same class of problems with a specific projection step used inside the Galerkin projection together with high-order quadrature rules and a time-step restriction. In this paper,
we provide a comprehensive comparison of the two methodologies mentioned, focusing on their theoretical properties and practical implementation aspects. We highlight shared foundational concepts and key differences of both approaches, with a particular focus on the selection of basis functions in the stochastic domain. As a highlight, we show that under specific conditions, the two formulations align, offering a unified framework that connects these distinct approaches. From our theoretical findings, we extend the development of high-order entropy conservative DG methods for the one-dimensional stochastic Galerkin shallow equations to two space dimensions; constructing entropy conservative two-point fluxes via primitive variables instead of entropy variables and applying it in our high-order DG setting.  
In numerical simulations, we verify and support our theoretical findings of a well-balanced and entropy-stable DG scheme which can be used to solve geophyiscal fluid flows with uncertainty. 
\end{abstract}

% Keywords

\maketitle

\section{Introduction}\label{sec_intro}

The shallow water (SW) equations, also referred to as the Saint-Venant equations, characterize the behavior of hydrostatic free-surface waves driven by gravity. These nonlinear hyperbolic balance laws are valid under the assumption that the water depth is much smaller than the wavelength, making them particularly suited to modeling a wide range of geophysical and environmental flow phenomena.
The SW equations are extensively used across various engineering and scientific disciplines, including river hydrodynamics, urban flood management~\cite{Soares-Frazao_etal_08}, tidal flow analysis~\cite{Kelly_etal_16}, storm surge prediction~\cite{Chen_etal_06}, and tsunami modeling~\cite{Varsoliwala_Singh_21}. 
As such, their numerical approximation remains an area of intense research where the demand is on high-order and structure-preserving methods, including a variety of different framework ranging from finite volume methods (FV) to finite element (FE) schemes, cf. \cite{ranocha2017, fjordholm2011, mantri2024fully, wu2021high, fu2022high, xing2014exactly, ciallella2025, zbMATH07342032,zbMATH07708330, abgrall2024} and references therein. These approaches are designed to produce simulations of real-world hydrodynamic systems that are reliable, physically accurate, and computationally efficient. High-order methods are especially advantageous, as they achieve lower errors even on relatively coarse grids. Currently, there is significant emphasis on structure-preserving techniques, which go beyond straightforward partial differential equation (PDE) discretization to maintain key physical or mathematical properties. In the context of the SW equations, such methods often focus on preserving stationary solutions (e.g., the lake-at-rest state), a feature known as the well-balancing property. They may also include entropy-conserving or dissipative schemes, as well as methods that ensure the non-negativity of the water height.
In high-order numerical discretization these implicit constraints have to be fulfilled simultaneously together with consistency conditions for solving the underlying PDE. 

To further enhance the reliability and predictive capability of SW simulations in real-world applications, comparison with experimental or observational data is increasingly sought. However, achieving such fidelity is complicated by the presence of uncertainties inherent to environmental modeling. Recently, the role of uncertainty quantification (UQ) in SW modeling has gained significant attention, as various sources of uncertainty can critically affect the accuracy of simulation outcomes. 
\revTwo{
These include uncertainties in the initial conditions, such as water height and velocity; measurement errors in bottom topography; and uncertainties in source terms, such as bottom friction or external forces. Separately, model simplifications introduce model-form uncertainty by creating discrepancies between the mathematical representation and physical reality.
}
Incorporating these uncertainties into the SW framework introduces several numerical and modeling challenges. These include maintaining conservation and well-balancing in the presence of uncertain inputs, accurately capturing the propagation of randomness through nonlinear and potentially chaotic dynamics, ensuring the positivity of water height under stochastic variations, and extending structure-preserving properties, such as equilibrium preservation, to stochastic settings. 

There exists a variety of methods to handle forward UQ for PDEs, where the most common approaches are either Monte-Carlo (MC) simulations, stochastic collocation (SC) or stochastic Galerkin (SG) ansatz, or variation of those,  cf. \cite{abgrall2017uq} and references therein. 
MC methods are widely used due to their simplicity of implementation and excellent scalability with many uncertainty sources. However, they suffer from very slow convergence rates, making them computationally expensive, especially for high-accuracy demands. SC methods improve upon this by offering faster convergence than MC and allowing the use of existing deterministic solvers without modification. Yet, SC methods are still limited by the curse of dimensionality, with computational cost growing rapidly as the number of uncertain parameters increases. 
\revTwo{
Here, we focus on the SG approach that allows for the simultaneous resolution of the physical and stochastic variables and provides a systematic framework for propagating uncertainty in the numerical solution.
}
While SG leads to larger coupled systems and typically requires more significant adaptations to existing solvers, it is well-suited for the development of structure-preserving, high-order schemes. In this context of hyperbolic equations, the SG approach was used  for scalar hyperbolic balance laws in~\cite{Pettersson_etal_09, Pettersson_16, oeffner20218stability, jin2016}. The treatment of hyperbolic systems of equations is much more involved due to the fact that a straightforward application of the SG approach may result in the loss of hyperbolicity and therefore to ill-posed problems \cite{despres2013}. Numerous recent efforts have been made to address this issue. Among those familiar to the authors are the following approaches:
\begin{enumerate}
    \item Transformation to Roe variables combined with Haar wavelets~\cite{Pettersson_etal_14, bender_2023, Gerster_Herty_20, Gerster_etal_22, pettersson2015polynomial},
    \item Limiting strategies~\cite{zbMATH07079486, duerrwachter2020},
    \item Filtering techniques~\cite{kusch2020filtered},
    \item Linearization methods~\cite{wu2017},
    \item Entropic variable representations~\cite{Poette_etal_09},
    \item Projection techniques together with high-order quadrature rules and time-step restrictions~\cite{dai_etal_2021, dai2023energy, Dai_etal_22, epshteyn2024energy}.
\end{enumerate}
Here, we will focus in particular on the approaches number one and six. 

 An entropy flux pair for the SG-SW equations using Roe variable transformation proposed in~\cite{Gerster_Herty_20} was picked up in \cite{bender_2023} to build an entropy-conservative  scheme for the SG-SW equation by constructing entropy-conservative two-point fluxes in the sense of Tadmor \cite{Tadmor_2003entropy}
and applied it in a flux-differencing ansatz within a discontinuous Galerkin spectral element method (DGSEM) framework \cite{chen2020review}.  An approach based on specific projection techniques in the SG ansatz together with high-order quadrature rules in the stochastic space, but with an additional time-step restriction was used in~\cite{dai_etal_2021,dai2023energy},
including an entropy pair and second-order entropy-conservative/dissipative schemes in one-space dimension \cite{dai2023energy}, which was later extended to two space-dimensions \cite{epshteyn2024energy}. Additionally, for constructing an entropy stable FV scheme, the approach from \cite{fjordholm2011} was adapted \cite{dai2023energy, epshteyn2024energy} which is also based on the application of entropy-conservative two-point fluxes \cite{Tadmor_2003entropy}. 
It appears to be unclear how the Roe variables formulation in~\cite{Gerster_Herty_20} is connected to the approach in \cite{dai2023energy}. However, we demonstrate in the following manuscript that, under certain assumptions on the basis functions, the approaches in \cite{Gerster_Herty_20, bender_2023} and \cite{dai2023energy} are \textbf{equivalent}.
Under these assumptions on the basis functions, the entropy-flux pairs used in \cite{dai2023energy} and \cite{bender_2023} coincide and we demonstrate the connection between both approaches.

The aim of this paper is to close open gaps in the community and clarify the connection. Additionally to this theoretical investigation, we extend the second-order entropy-conservative/dissipative method in~\cite{epshteyn2024energy} to arbitrarily high-order using the DG framework.
\revOne{So, the present work does not introduce a new stochastic approximation basis or stochastic 
projection methodology. 
The stochastic Galerkin representation employed herein follows established formulations using Haar wavelets from \cite{Gerster_Herty_20, Gerster_etal_22}.
We develop a high-order, entropy stable DG spatial discretization that is provably well-balanced.
If a scheme is not well-balanced, then spurious waves on the order of the mesh spacing can be created and propagate throughout the domain. These artificial waves are indistinguishable from physical waves and, in the worst case, may lead to breakdown of the solution.
The present work
serves to unify existing SG formulations and restrict the verification study to physical-space convergence.
However, the numerical testing of the well-balanced property and solution quality and behavior in the presence of discontinuities is considered at different resolutions in both the stochastic and physical basis expansions.
This baseline entropy stable and high-order method could be extended to novel shock-capturing or limiting strategies, but this is left as future work.}

%Therefore, 
The paper is organized as follows: 
In Section \ref{se_se}, we explain the spectral expansion in random variables, introduce the stochastic Galerkin projection and explain the influence of the different bases on the properties of the spectral Galerkin projections. Next, in Section \ref{se_SGSW} we introduce the stochastic Galerkin shallow water equations following \cite{dai_etal_2021} and \cite{Gerster_Herty_20} respectively, then in subsection \ref{subsec_relation}, we point out the relation between the different SG-SW formulations and demonstrate their  equivalence if an adequate basis is used. In Section  \ref{sec_DGmethod}, we introduce the entropy-conservative DGSEM scheme and compare  and discuss the entropy-conservative fluxes developed in \cite{bender_2023} and \cite{dai2023energy}. Here, we shortly explain the general procedure which we use to construct a two-point flux for the SG-SW formulation in two space dimension which differs however from the version proposed in \cite{epshteyn2024energy}. Afterwards, in Section \ref{se_numerics} we demonstrate in a proof of concept all the derived properties by several numerical examples.
A summary and outlook finish the manuscript where in the Appendix \ref{sec_Appendix} the lengthy  calculations and proofs can be found for the main theorems together with  a list of notation for completeness.

\input{Spectral_expansions}

\input{SW}

\input{DG}

\input{Numerical_results}

\section{Conclusions}

We have  provided a detailed comparative study of two recent stochastic Galerkin approaches for the shallow water equations, highlighting both their shared theoretical foundations and their methodological differences, particularly in the choice of stochastic basis functions and numerical flux formulations. We demonstrate that under specific basis assumptions, the Roe-variable-based and specific project-based formulations can be reconciled, offering a unified framework for stochastic Galerkin methods. Building on this insight, we developed a high-order entropy-conservative (dissipative)  DG method to two-dimensional problems, introducing flexible flux constructions and high-order approximations that maintain entropy conservation and well-balanced properties. Our numerical experiments confirm the theoretical developments and offer practical guidance for solving geophysical fluid flows under uncertainty, providing valuable tools for both the numerical PDE and UQ communities.

In future work, we aim to address several challenges. First, in the stochastic Galerkin approach with many random variables, most coefficients in the flux matrix have little or no impact on the expected value. We plan to integrate our methods with adaptive techniques to eliminate contributions below a certain threshold, drawing connections to model order reduction approaches. Additionally, we will incorporate limiting strategies in the stochastic setting to suppress spurious oscillations. 
\revOne{Related to shock capturing is the issue of positivity as the water height goes to zero near the bottom topography, i.e., wetting and drying.
The Haar wavelet stochastic basis are positive and free of overshoots; however, the high-order polynomial basis in the physical space is not.
Gibbs phenomena may lead to negative water heights and unstable solution behavior.
To create a shock capturing method with wetting and drying capabilities may require careful limiter design within the time stepping algorithm or in the projection step between physical and stochastic space.}
Finally, investigating well-balancing for moving equilibria remains an important and meaningful direction for further study.

\printcredits

\section*{Data Availability}
A reproducibility repository with necessary instructions and code to reproduce the presented numerical results are available in \cite{oeffner2026numericalRepro}.
	
\section*{Declaration of competing interest}
The authors declare that they have no known competing financial interests or personal relationships that could have appeared to influence the work reported in this paper.
	
\section*{Acknowledgments}
    The work of Philipp \"Offner was supported by the German Research Foundation (DFG) within the framework of the priority research program SPP 2410 under the grant  project OE 661/5-1 (525866748) and under the personal grant  OE 661/4-1(520756621). 
    Andrew Winters thanks Henry Haase for his help optimizing the implementation of the stochastic Galerkin in the Trixi.jl framework.

\appendix
\input{Appendix}
  
\bibliographystyle{plain}
\bibliography{bibliography.bib}

\end{document}

%% file: Spectral_expansions.tex
\section{Spectral expansions in random variables}\label{se_se}

Random functions with finite variance can be represented as series expansions in independent random variables with known distributions, which is the basis for the polynomial chaos framework~\cite{Ghanem_Spanos_91} and a wide variety of subsequent methods. In this way, unknown functions of known random variables, e.g., PDE solutions with random input parameters, can be efficiently parameterized and eventually computed to good accuracy. Consider a vector of random variables $\xi = (\xi_1, \dots, \xi_{n}) \in \mathbb{R}^{n}$ with product-type probability density function (PDF) $\rho(\xi)=\Pi_{i=1}^{n} \rho_{i}(\xi_{i})$. 
We will consider finite-variance functions of $\xi$, i.e., functions on the 
space 
\begin{equation}
\label{eq:wtd_space}
L^2_{\rho} = \left\{ z: \mathbb{R}^{n} \mapsto \mathbb{R} \,\middle|\, \int_{\mathbb{R}^n} z(s)^2 \rho(s)\textup{d}s < \infty \right\},
\end{equation}
with the associated inner product
\[
\left\langle y, z\right\rangle \equiv \int_{\mathbb{R}^{n}} y(s)z(s) \rho(s)\text{d}s.
\]
Furthermore, let $\{ \psi_{k}(\xi)\}_{k=1}^{\infty}$ be an orthonormal basis in $\xi$ on the space~\eqref{eq:wtd_space},
i.e.,
\[
\left\langle \psi_{k}(\xi) , \psi_{l}(\xi)\right\rangle = \delta_{kl}, \quad \delta_{kl}=\left\{
\begin{array}{ll}
1 & k=l\\
0 & k \neq l
\end{array}
\right. , \quad \forall k,l \in \mathbb{N},
\]
where $\delta_{kl}$ denotes the Kronecker delta. Any stochastic function $v \in L^2_{\rho}$, e.g., a PDE solution where input uncertainties are parameterized as functions of $\xi$, can be represented by
a series expansion in $\{ \psi_{k}(\xi)\}_{k=1}^{\infty}$, 
\[
v(\xi) = \sum_{k=1}^{\infty} \hat{v}_k \psi_{k}(\xi),
\]
with deterministic coefficients $\{ \hat{v}_k \}_{k=1}^{\infty}$ and where the identity holds in the $L^2_{\rho}$ sense. 
 Common choices of basis functions include classical orthogonal polynomials from the Askey scheme~\cite{Askey_Wilson_85}, as introduced in the seminal work~\cite{Xiu_Karniadakis_02} and since then widely used in several applications \cite{review}. Orthogonal polynomials have attractive approximation properties and are efficient in representing random variables with distributions being similar (in the sense being close to linear) to the assumed distributions of the random variables $\xi$. Orthogonal polynomials are however not suitable for nonsmooth functions, due to, e.g., Gibbs oscillations close to discontinuities~\cite{Poette_etal_09}. For these problems, basis functions that are localized in space have been proposed, e.g., multi-element generalized polynomial chaos that consists of classical orthogonal polynomials restricted to subsets of the stochastic domain~\cite{Wan_Karniadakis_05}. Wavelets provide robust and localized hierarchical bases using piecewise linear Wiener-Haar expansions~\cite{LeMaitre_etal_04a} or piecewise polynomial multi-wavelets~\cite{LeMaitre_etal_04b}.

 \begin{figure}[ht]
    \centering 
\includegraphics[width=0.98\textwidth]{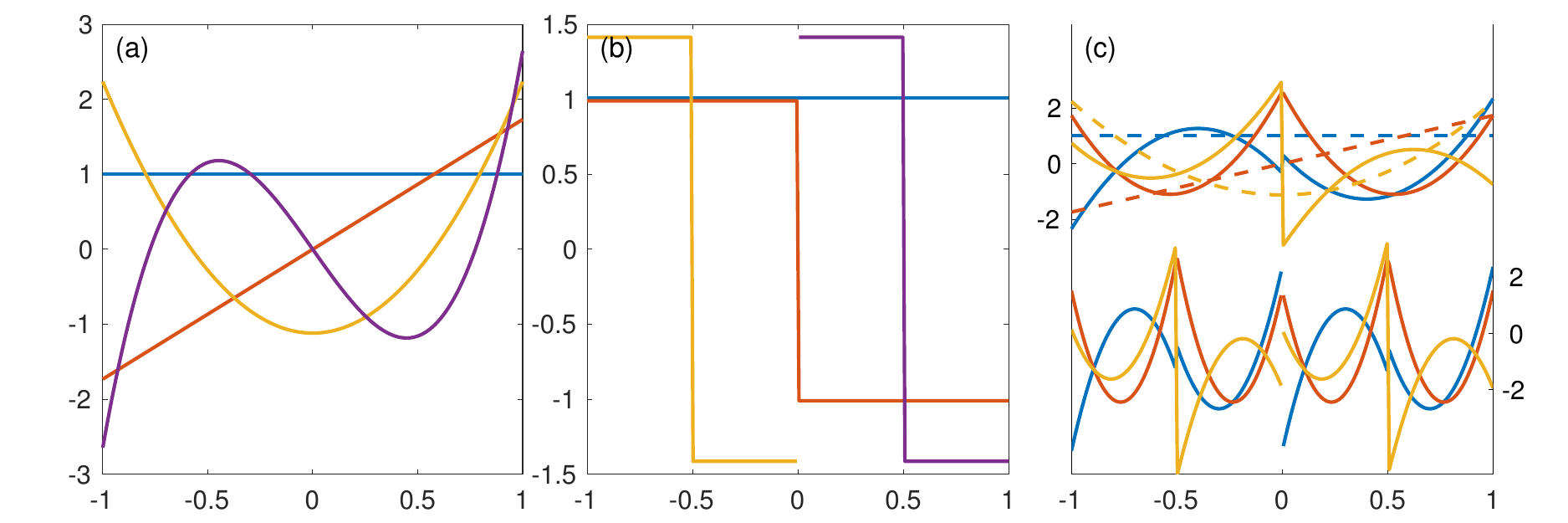}
\caption{The first few Legendre polynomials (a), Haar wavelets (b), and piecewise quadratic multi-wavelets (c).}
\label{fig:Lege_vs_Haar_exp}
\end{figure}

In many situations, one wants to approximate an infinite-dimensional random field which for practical purposes requires the truncation of both the number of random dimensions to a finite number $n$, and truncation of the basis to a finite number of functions, denoted $K$. For a random function with spatial and temporal dependence, we consider the truncated spectral expansion
\begin{equation}
\label{eq:trunc_sp_exp}
v(x,t,\xi) = \sum_{k=1}^{K} \hat{v}_k(x,t) \psi_{k}(\xi),
\end{equation}
where $x=(x_1,\dots, x_d)\in \mathbb{R}^{d}$ denotes spatial coordinates in $d$ dimensions, and $t$ denotes time.
In polynomial bases, the number $K$ depends on the method of truncation, e.g., with polynomial order $N_p$, $K=(N_p+1)^n$ for a tensor basis, and $K=(N_p + n)!/(N_p! n!)$ for a total-order basis where all multivaraiate polynomials with at most and including total degree $N_p$ are included. An even more truncated basis suitable for problems with limited contributions from mixed polynomials is the
hyperbolic index set~\cite{Blatman_Sudret_11}.

The number of basis functions $K$ in a univariate expansion in multi-wavelets is $K=(N_{p}+1) 2^{N_r}$, where $N_p$ is the piecewise polynomial order and $N_r$ is the number of resolution levels. In a multivariate expansion in $n$ random dimensions, the number of basis functions is $K=(N_{p}+1) 2^{N_r n}$, assuming a tensor basis with the same order of resolution levels and polynomial order in all dimensions. The tensor basis grows quickly with the number of dimensions and resolution levels. With other choices, e.g., total order bases, fewer functions are included in the truncated basis, but they may lead to representations that fail to maintain physically relevant properties~\cite{Shaw_etal_20} as well as attractive numerical properties, e.g., semi-analytical eigenvalue decompositions.

\subsection{Stochastic Galerkin projection}
Projection of a physical model with random inputs onto a truncated basis leads to a finite and deterministic set of mathematical expressions for the spectral expansion coefficients. In essentially all nonlinear and many linear models, products between random objects need to be appropriately represented.
For efficient projection onto the stochastic basis, it will be useful to define the stochastic Galerkin matrix $\M(\hat{v})$ that takes as argument a vector of spectral coefficients $\hat{v} \in \mathbb{R}^K$ and is defined by 
\begin{equation}
\label{eq:SG_matrix}
[\M(\hat{v})]_{ij} \equiv \sum_{k=1}^{K}\hat{v}_{k}\langle \psi_{i}\psi_{j},\psi_{k} \rangle 
\mbox{ for } i,j=1,\dots, K,\quad  \forall \hat{v} \in \mathbb{R}^{K}.
\end{equation}
With a tensor product basis, we have that
\begin{equation}
\label{eq:SG_matrix_tens_struct}
\M = \M_{1} \otimes \M_{2} \otimes \dots \otimes \M_{n},
\end{equation}
where $\M_m$ ($m=1,\dots, n$) denotes the univariate stochastic Galerkin matrix in random variable $m$. Typically, attractive properties of the matrices $\M_{m}$ can also be shown to hold for the full matrix $\M$. For all $\hat{v},\hat{w} \in \mathbb{R}^{K}$, it holds that
\begin{equation}
\label{eq:SG-vec-comm}
\M(\hat{v})\hat{w} = \M(\hat{w})\hat{v}. 
\end{equation}

\subsection{Properties of different bases}\label{subsec_pro}
Orthogonal polynomials bases, e.g., Hermite or Legendre polynomials, can lead to unphysical properties in the truncated expansions of certain functions. Examples include oscillations around a discontinuity in stochastic space~\cite{Poette_etal_09}, and the representation of a lognormal diffusion coefficient in Hermite polynomials, unless an extended basis is employed~\cite{Pettersson_etal_13}. Next, we make these observation more concrete in the current setting of a nonlinear hyperbolic PDE, and show that orthogonal polynomials can lead to failure in a numerical solver already during the initial projection, or at a later time even if all original projections remain physical. We do not prove that a Haar wavelet basis will always converge, but we do show that it is not subject to the same pitfalls demonstrated for the orthogonal polynomial basis. 

Figure~\ref{fig:Lege_vs_Haar_exp} shows approximation in Legendre polynomials of functions which should be nonnegative for $\xi \in [-1, 1]$ but fails due to non-smooth features. In contrast, Haar wavelets never yields over- or undershoots in the sense that any level of projection of any function will not yield extreme values beyond those of the function itself. To prove this claim, let $f_{\text{Haar}}$ be the Haar wavelet projection of any $L_2$ function $f(\xi)$ for $\xi \in U[-1,1]$. For any fixed $\xi'$, $f_{\text{Haar}}(\xi')$ is constant in a neighborhood of $\xi'$, denoted $D'$, i.e., $f_{\text{Haar}}(\xi) = f_{\text{Haar}}^{D'}$ for all $\xi \in D'$. By construction, $
f_{\text{Haar}}^{D'} = \frac{1}{|D'|} \int_{D'} f(s)\textup{d}s .
$
Also,
\[
\min_{\xi \in D'}f(\xi) \leq
\frac{1}{|D'|} \int_{D'} f(s)\textup{d}s 
\leq \max_{\xi \in D'}f(\xi). 
\]
Hence, $f_{\text{Haar}}(\xi)$ cannot over- or undershoot the true value $f(\xi)$ on $D'$, and since $\xi'$ (and hence $D'$) was arbitrary, it cannot happen anywhere in the range of $\xi$.
This shows that the projection of, e.g., initial and boundary conditions or material parameters cannot lead to unphysical values with a Haar basis. As an illustration, Figure~\ref{fig:Lege_vs_Haar_exp_v2} (a) and (b) shows approximations of the functions $f_1, f_2$ given by
\begin{equation}
\label{eq:example_fcns_gPC}
f_1(\xi) = 1-|\xi|, \quad f_2(\xi)=\left\{
\begin{array}{ll} 
1 & \mbox{if } \xi \leq 0\\
0.02 & \mbox{if } \xi > 0
\end{array}\right..
\end{equation}
Due to the low regularity of $f_1$ and $f_2$, approximation with Legendre polynomials leads to values outside the range of the true functions in both cases. The Haar wavelet approximations do not suffer from this phenomenon but the results show that unless the target function is piecewise constant with discontinuities coinciding with the discontinuities of the basis functions, the convergence can be slow.

\begin{figure}[ht]
    \centering 
\includegraphics[width=0.98\textwidth]{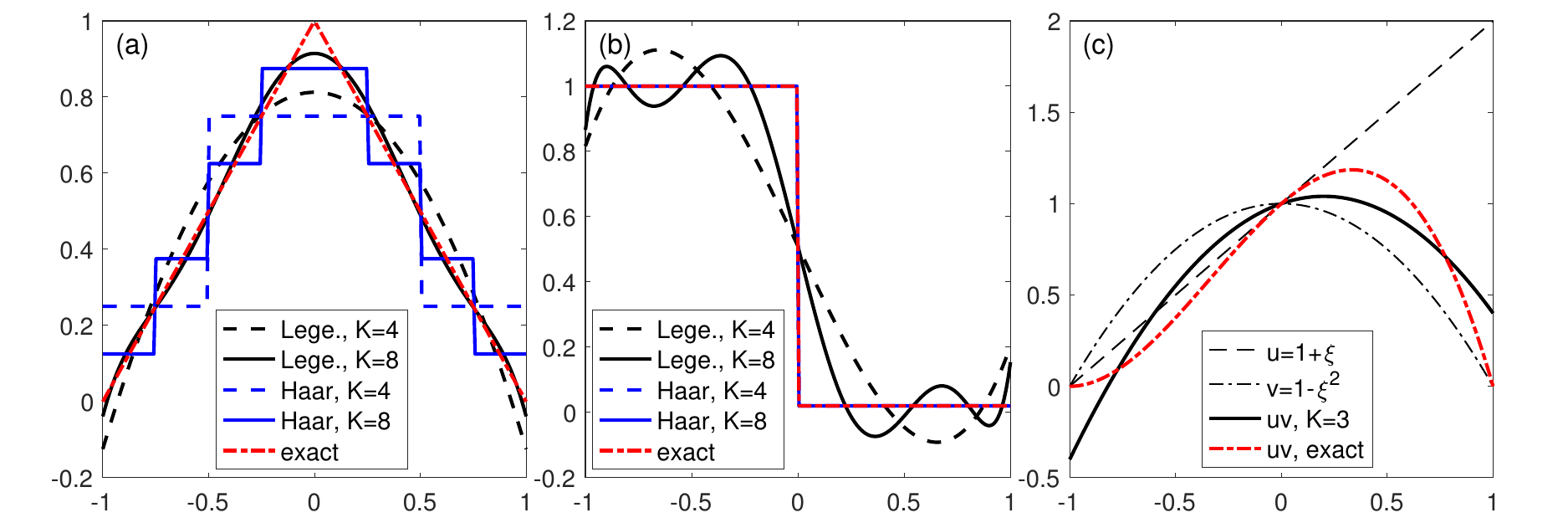}
\caption{Legendre polynomial and Haar wavelet reconstruction of the nonsmooth functions $f_1$ (a), $f_2$ (b). The factors $(1+\xi)$ and $(1-\xi^2)$ can be exactly represented with $K=3$ Legendre polynomials, but the product $(1+\xi)(1-\xi^2)$ projected onto the same basis assumes negative values with non-zero probability (c).}
\label{fig:Lege_vs_Haar_exp_v2}
\end{figure}
Even in the case where a polynomial basis can exactly represent for instance initial conditions, unphysical behavior may be introduced over time. An example is shown in Figure~\ref{fig:Lege_vs_Haar_exp_v2}~(c), where the functions $u=1+\xi$ and $v=1-\xi^2$ can be exactly represented by a $K=3$ order expansion in Legendre polynomials. However, the product $u(\xi)v(\xi)$ cannot be exactly represented by a $K=3$ Legendre expansion, leading to negative values of a pseudospectral product of otherwise non-negative factors.
Next, we will show that unphysical features in the sense of values beyond the ranges of the true function can also not emerge over time in a stochastic Galerkin setting with Haar wavelet bases, assuming that all fluxes can be obtained by products of random variables. Consider two arbitrary Haar wavelet projections $u_{\text{Haar}}$ and $v_{\text{Haar}}$ that are within physical bounds of some underlying exact functions $u(\xi)$ and $v(\xi)$ (induction hypothesis). Hence the product $u_{\text{haar}} v_{\text{haar}}$ must be within the physical bounds of $u(\xi) v(\xi)$. Moreover, it holds pointwise for the product that $u_{\text{Haar}} v_{\text{Haar}} = (uv)_{\text{Haar}}$, so there is no projection error. Hence, $(uv)_{\text{Haar}}$ will also be within the physical bounds. \revTwo{The preceding theoretical analysis serves to motivate the choice of Haar-type bases among other things as a means to isolate spurious oscillations that may otherwise stem from both the numerical scheme and the choice of polynomial basis, and will hence be relevant in the numerical results.}

If the stochastic Galerkin matrix $\M(\cdot)$ has an eigenvalue decomposition with constant eigenvectors, as is the case with Haar wavelets and piecewise linear multi-wavelets~\cite{Pettersson_etal_14, pettersson2015polynomial}, the hyperbolic SG PDE systems often become amenable to theoretical analysis, as demonstrated in, e.g.,~\cite{pettersson2015polynomial,Gerster_Herty_20,Gerster_etal_22}. The following property and the subsequent identities will be central in the proofs of the current paper.

\noindent
\textbf{Property 1:} The stochastic Galerkin matrix $\M(\cdot)$ defined by~\eqref{eq:SG_matrix} has an eigenvalue decomposition with constant eigenvectors $V \in \mathbb{R}^{K \times K}$, i.e.,
\begin{equation}
\label{eq:basis_P1}
\M(\hat{v}) = V \Lambda_{\hat{v}} V^T, \quad \forall \hat{v} \in \mathbb{R}^{K}.
\tag{P1}
\end{equation}
The following properties hold for bases for which~\eqref{eq:basis_P1} is true, but typically not for other bases. 
\begin{lemma}
For any basis satisfying~\eqref{eq:basis_P1}, the following is true for all $\hat{v},\hat{w} \in \mathbb{R}^{K}$: 
\begin{itemize}
\item[(i)] The matrices $\M(\hat{v})$ and $\M(\hat{w})$ commute.
\item[(ii)] $\M(\M(\hat{v})\hat{w})  = \M(\hat{v}) \M(\hat{w})$.
\item[(iii)] \revTwo{$\lambda_j(\hat{v})\neq 0$, $j=1,\dots, K$, then} $\M^{-1}(\hat{v}) = \M(\hat{v}^{-1})$ where $\hat{v}^{-1} \equiv \M^{-1}(\hat{v})e_1=V \Lambda_{\hat{v}}^{-1}V^T e_1$.
\item[(iv)] 
\revTwo{If $\lambda_j(\hat{v}) > 0$, $j=1,\dots, K$, then} $\hat{v}^{1/2} = V \Lambda_{\hat{v}}^{-1/2} V^T \hat{v}$.
\end{itemize}
\end{lemma}
\begin{proof}
Property (i) is proven in~\cite[Lemma 4.1]{Gerster_Herty_20}, and property (ii) in~\cite{Gerster_etal_22}. For (iii), $\M(\hat{v})\hat{v}^{-1}=e_1$ by definition, hence $\M(\M(\hat{v})\hat{v}^{-1}) = \M(\hat{v})\M(\hat{v}^{-1}) = \M(e_1)=I$, from which follows that $\M^{-1}(\hat{v}) = \M(\hat{v}^{-1})$.  For Property (iv), we have that $\hat{v}^{1/2}$ must satisfy $\M(\hat{v}^{1/2})\hat{v}^{1/2}=\hat{v}$, hence 
$\M(\hat{v}^{1/2})\M(\hat{v}^{1/2})=\M(\hat{v})$. Due to the orthogonality of the constant eigenvectors, it follows the last identity that the unknown eigenvalues of $\M(\hat{v})$ satisfy $(\Lambda_{\hat{v}^{1/2}})^{2} = \Lambda_{\hat{v}}$, assuming that $\lambda_j(\hat{v})> 0$ for all $j=1,\dots, K$. Again using orthogonality of the eigenvectors, it follows from $\M(\hat{v}^{1/2})\hat{v}^{1/2} = V\Lambda_{\hat{v}}^{1/2}V^T \hat{v}^{1/2} = \hat{v}$ that $\hat{v}^{1/2}= V \Lambda_{\hat{v}}^{-1/2} V^T \hat{v}$.
\end{proof}

%% file: SW.tex
\section{Stochastic Galerkin shallow water equations}\label{se_SGSW}
In this work, the model under consideration is the SW equations, also known as Saint-Venant equations, which characterize the behavior of hydrostatic free surface waves influenced by gravity. The SW equations are used to describe tsunamis or urban flood behavior for example. The SW system can be derived from the Navier-Stokes equations and the deterministic system is defined in one space dimension with coordinate $x$ via 
\begin{equation}\label{eq_SW_pure}
\frac{\partial}{\partial t} \begin{pmatrix}
h \\
q
\end{pmatrix} + \frac{\partial}{\partial x} \begin{pmatrix}
q \\
q^2/h + \frac{1}{2} gh^2
\end{pmatrix}
= \begin{pmatrix}
0 \\
-gh \frac{\partial}{\partial x}b 
\end{pmatrix},
\end{equation} 
where $h=h(x,t)$ denotes water height, $q=q(x,t)$ the  discharge, $g$ the gravitational constant and $b=b(x)$ the bottom topography. In addition, $q\equiv h v$ with velocity $v$. 
The first equation in \eqref{eq_SW_pure} represents the mass conservation, while the second equation represents the momentum balance. The right-hand side of \eqref{eq_SW_pure} represents the source terms. We assume a frictionless environment with time-independent bottom topography. As long as the height of the water is nonnegative, the system \eqref{eq_SW_pure} belongs to the class of hyperbolic balance laws. The conserved vector is given by $u\equiv(h, q)^T$.

As described in Section~\ref{sec_intro},
it is reasonable to address~\eqref{eq_SW_pure} within a framework that accounts for uncertainties. Therefore, in our context the quantities $u$ and $b$ depend not only on space and time, but also on the vector of random variables $\xi$ which results in the stochastic shallow water equations, e.g. $u= u(x,t, \xi)$ and $b=b(x,t, \xi)$. We apply the generalized polynomial chaos approach as introduced in Section \ref{se_se}.
By using a truncated spectral expansion \eqref{eq:trunc_sp_exp} and a stochastic Galerkin approach via basis functions $\psi_j$, the stochastic SW equations are re-written as a purely deterministic system of equations where the spectral coefficients $\hu$ have to be calculated. 
However, the properties of the system highly depend on the selected basis and the projection used. As established in \cite{pettersson2015polynomial}, employing orthogonal polynomials from the Askey scheme through a classical projection method can result in a loss of hyperbolicity
and several works can be found in the literature \cite{dai_etal_2021,zbMATH07079486,Pettersson_etal_14} where this problem is addressed. 
The first approach to solve this issue has been presented up-to-the-authors knowledge in \cite{Pettersson_etal_14} by using the Roe variables and the Wiener-Haar expansion instead. 
As described in Subsection \ref{subsec_pro} the heuristic  explanation for this is that using orthogonal polynomials does not ensure the physical properties in the truncated expansion of certain function. In the context of shallow water, the non-negativity of the water height can not be ensured if orthogonal polynomials have been used, whereas with the Haar-wavelets this can be guaranteed, cf. Figure \ref{fig:Lege_vs_Haar_exp_v2}. 
Meanwhile, in \cite{dai_etal_2021} another hyperbolicity-preserving 
approach has been presented and further applied in \cite{Dai_etal_22,dai2023energy}. Here, the idea is to select carefully the polynomial chaos expansion in conserved variables to avoid issues with the positivity of the water height. In our manuscript, we compare both approaches, highlighting their common features and interrelationships. This comparison should enhance understanding of the differences and assist researchers in future applications.
Therefore, we shortly introduce the two considered ideas. 

\subsection{Hyperbolic preserving SG-SW formulation}
Using the same approach as in~\cite{Dai_etal_22} to perform stochastic Galerkin projection of~\eqref{eq_SW_pure} with the approximation $\M(\hq)\M^{-1}(\hh)\hq$ of the nonlinear flux term $q^2/h$, leads to the hyperbolic system
\begin{equation}
\label{eq:SG_SWE_cons} 
\frac{\partial}{\partial t}
\begin{pmatrix}
\hh\\
\hq
\end{pmatrix}
+
\frac{\partial}{\partial x}
\begin{pmatrix}
\hq\\
\M(\hq) \M^{-1}(\hh)  \hq 
+ \frac{1}{2}g  \M(\hh)  \hh
\end{pmatrix}
=
\begin{pmatrix}
\hat{0} \\
-g \M(\hh) \frac{\partial}{\partial x} \hat{b}
\end{pmatrix},
\end{equation}
with the flux Jacobian
\begin{equation}
\label{eq:SG_Jacobian_cons}
J_{\text{C}}(\hh,\hq) = \begin{pmatrix}
O_{K} & I_{K} \\
g\M(\hh)-\M(\hq)\M^{-1}(\hh)\M(\M^{-1}(\hh)\hq) & \M(\hq)\M^{-1}(\hh
) + \M(\M^{-1}(\hh)\hq)
\end{pmatrix},
\end{equation}
where $O_{K}$ is the zero matrix and $I_{K}$ is the identity matrix, both of size $K \times K$. Here, $\hat{\bullet}$ denotes again the $K$-length coefficient vector and the derivatives are used componentwise. 
We rewrite~\eqref{eq:SG_SWE_cons} using $\hu= \begin{pmatrix}
\hh\\
\hq
\end{pmatrix}$, $\hat{F}= \begin{pmatrix}
\hq\\
\M(\hq) \M^{-1}(\hh)  \hq 
+ \frac{1}{2}g  \M(\hh)  \hh
\end{pmatrix} $ and $\hat{S}= \begin{pmatrix}
\hat{0} \\
-g \M(\hh) \frac{\partial}{\partial x} \hat{b}
\end{pmatrix} $
 via 
\begin{equation}\label{eq:SG_SWE_cons_2}
\frac{\partial}{\partial t}
\hu
+
\frac{\partial}{\partial x}
\hat{F}
=
\hat{S}
\end{equation}
which is a strictly deterministic hyperbolic balance law. As presented in~\cite{dai2023energy}, the SG-SW formulation~\eqref{eq:SG_SWE_cons} is equipped with an entropy-entropy flux pair $(\hat{\eta}, \hat{H})$ defined by 
\begin{equation}\label{eq_ent_flux_dai}
\begin{aligned}
\hat{\eta}(\hh,\hq) &\equiv \frac{1}{2}\left( \hq^{T}\M^{-1}(\hh)\hq + g \hh^{T}\hh\right) + g\hh^{T}\hat{b},\\
\hat{H}(\hh,\hq) &\equiv \frac{1}{2} \hv^{T}\M(\hq) \hv + g\hq^{T}\hh + g\hq^{T}\hat{b},  
\end{aligned}
\end{equation} 
with $\hv= \M^{-1}(\hh) \hq$.
As demonstrated in detail in~\cite[Remark 3.1]{dai2023energy}, the entropy variables are given by 
\begin{equation}\label{eq:entropy_variable_Dai}
\hat{w}= \left(\frac{\partial \hat{\eta}}{\partial \hu} \right)^T= \left( -\frac{1}{2}\hv^T\M(\hv)+g (\hh+\hat{b})^T, \hv^T \right),
\end{equation}
and the entropy potential is given by
\begin{equation}\label{eq_potential}
\Phi= \hat{w} \hat{F }- \hat{H} = \frac{1}{2} g \hv^T \M( \hh) \hh.
\end{equation}

\subsection{SG-SW formulation via Roe variable transformation}

Following the approach taken in~\cite{bender_2023,Gerster_Herty_20},
Eq.~\eqref{eq_SW_pure} can be expressed in terms of the Roe variables $\mathfrak{r}=(\alpha, \beta)$, where $\alpha=\sqrt{h}$ and $\beta=\sqrt{h} v$ with speed $v=\frac{q}{h}$. Stochastic Galerkin projection of~\eqref{eq_SW_pure} with Roe variables leads to a system of equations in the vectors of gPC coefficients $\hat{\alpha}, \hat{\beta}$. Defining $\hh^{1/2} \equiv \hat{\alpha} $ as the solution\footnote{We assume in the following that the solution always exists which from a practical point is meaningful; for a more detailed discussion, cf. \cite{Gerster_etal_22}.} to  
\begin{equation}\label{eq_Roe_transform}
\begin{aligned}
\M(\hat{\alpha})\hat{\alpha} = \hh, 
\text{ and } \hat{\beta} &\equiv \M^{-1}(\hh^{1/2})\hq,
\end{aligned}
\end{equation}
 the Roe variable stochastic Galerkin formulation of~\eqref{eq_SW_pure} in conservative variables becomes: 
\begin{equation}
\label{eq:SG_SWE_Roe}    
\frac{\partial}{\partial t}
\begin{pmatrix}
\hh\\
\hq
\end{pmatrix}
+
\frac{\partial}{\partial x}
\begin{pmatrix}
\hq\\
\M(\M^{-1}(\hh^{1/2})  \hq) \M^{-1}(\hh^{1/2})  \hq 
+ \M(\M(\hh^{1/2})  \hh^{1/2}) \M(\hh^{1/2})  \hh^{1/2}
\end{pmatrix}
=
\begin{pmatrix}
\hat{0} \\
-g \M(\hh) \frac{\partial}{\partial x} \hat{b}
\end{pmatrix},
\end{equation}

Note that contrary to the strategy in \cite{bender_2023,Gerster_Herty_20}, we express the Roe variable SG-SW equations in conservative variables. This is done to facilitate comparison with the formulation \eqref{eq:SG_SWE_cons}, and avoid lengthy transformations between the variables in the derivations presented in the manuscript.
Indeed, for the hyperbolic balance law \eqref{eq:SG_SWE_Roe}, one finds entropy-entropy flux pairs as demonstrated and described in \cite{Gerster_Herty_20, bender_2023}. They are given via 

\begin{equation}\label{eq_ent_flux_gerster}
\begin{aligned}
\hat{\eta}_{\text{Roe}}&= \frac{1}{2}\left( \hat{\alpha}^T \M(\hat{\alpha})^T (\M(\hat{\alpha}) \hat{\alpha})+ \hat{\beta}^T\hat{\beta} \right)+ g\hat{\alpha}^T \M(\hat{\alpha}) \hat{b},\\
\hat{H}_{\text{Roe}}&= \frac{1}{2} \hat{\beta}^T \M(\hat{\beta})\M^{-1}(\hat{\alpha}) \hat{\beta} + g \hat{\alpha}^T \M(\hat{\alpha})^2  \hat{\beta}+  g \hat{\beta}^T \M(\hat{\alpha}) \hat{b},
\end{aligned}
\end{equation}
where we now use the identity~\eqref{eq_Roe_transform} with corresponding entropy variables and potential. 
In the following, we will describe how the equations \eqref{eq:SG_SWE_cons} and \eqref{eq:SG_SWE_Roe} 
are connected and what is the relation between the entropy-entropy flux pairs  \eqref{eq_ent_flux_dai} and \eqref{eq_ent_flux_gerster}.
Before that, we describe two  essential properties of the SG-SW equations, which are central to this manuscript and have attracted considerable attention in recent years, namely entropy stability and the existence of non-trivial steady-state solutions. Numerical schemes designed to preserve steady states, often referred to as well-balanced schemes, are of particular interest for capturing steady solutions accurately.
The simplest and most widely studied steady-state solution is the "lake at rest"- scenario, defined as  
\begin{equation}\label{eq_well_balanced}
 v = 0, \quad h+b \equiv h^0+b^0 \in \mathbb{R}_0^+, \quad \forall x \in \Omega, \, \forall t \in [0, T].
\end{equation}
In \eqref{eq_well_balanced} the superscript refers to the initial data. 
These properties can be further translated to the SG-SW framework assumimg that all 
velocity coefficients $\hv_k$ are zero and the coefficients of the water surface remain componentwise constant, i.e. 
\begin{equation}\label{eq_well_balanced_SW}
 \hv = 0, \quad \hh_0+\hat{b}_0 \equiv \hh_{0}^0+\hat{b}_{0}^0 \in \mathbb{R}_0^+, \quad \forall x \in \Omega, \, \forall t \in [0, T].
\end{equation}
Again, in \eqref{eq_well_balanced_SW} the superscript refer to the initial data, and the subscript index refers to the zero coefficient from the spectral expansion. 

The other property which we are interested in is entropy conservation and entropy stability. It means that the solutions of the stochastic SW equations \eqref{eq:SG_SWE_cons_2} or  \eqref{eq:SG_SWE_Roe} in addition fulfills the implicit entropy (in)equality constraint, e.g. the solution $\hu$ of \eqref{eq:SG_SWE_cons_2} satisfies the (in)equality 
\begin{equation}\label{eq:ent_ineq}
\frac{\partial}{\partial t}\hat{\eta}(\hu)+\frac{\partial}{\partial x} \hat{H} \stackrel{(\leq)}{=} 0.
\end{equation}
Analogously for SG-SW equations \eqref{eq:SG_SWE_Roe} in Roe variables with their corresponding entropy variable pair \eqref{eq_ent_flux_gerster}. 
\begin{remark}
The reason to consider entropy solutions for hyperbolic conservation laws (this means the source term is equal to zero) and balance laws lie 
in the fact, that weak solutions of the system are in general not unique. Uniqueness is obtained for entropy solutions at least for scalar multi-dimensional problems and one-dimensional systems (under certain assumptions), we refer for the scalar case to \cite{kruzhkow1970} and for the system case to recent publications \cite{bressan2024,amadori2002,zbMATH07834139} and references therein. 
Note that the problem for multidimensional system is different. Here, several non-uniqueness results are known, cf. \cite{zbMATH07921311,zbMATH07191485} and references therein. 
Therefore, more general solution forms are often used, but to show their existence, consistent structure-preserving numerical methods are often used. It is therefore also essential that the solutions fulfill such inequalities so that they are physically meaningful, cf. \cite{feireisl2021} for an overview. 
\end{remark}

\subsection{Relation between the different SG-SW formulations}\label{subsec_relation}
Next, we investigate in more detail the formulations~\eqref{eq:SG_SWE_cons} and~\eqref{eq:SG_SWE_Roe}, and the corresponding entropy-entropy flux pairs. 
We note that the formulations~\eqref{eq:SG_SWE_cons} and~\eqref{eq:SG_SWE_Roe} are not equivalent in general. However, for an important class of stochastic basis functions, the two formulations are indeed identical.  We demonstrate the following Theorem: 
\begin{theorem}\label{th_sw}
For any stochastic basis satisfying~\eqref{eq:basis_P1}, the SG-SW formulations~\eqref{eq:SG_SWE_cons} and~\eqref{eq:SG_SWE_Roe} are identical.
\end{theorem}
\begin{proof}
It suffices to show that the  terms in the second entries of the flux functions are identical. We have that
\begin{equation}
\begin{split}
\M(\M^{-1}(\hh^{1/2})  \hq) \M^{-1}(\hh^{1/2})  \hq 
& =
\M(\hq)\M^{-1}(\hh^{1/2}) \M^{-1}(\hh^{1/2}) \hq
=
\M(\hq)
\M(\M(\hh^{-1/2})\hh^{-1/2})
\hq \\
& \overset{\mbox{by def.}}{=}
\M(\hq)
\M(\hh^{-1})
\hq 
=
\M(\hq)
\M^{-1}(\hh)
\hq, 
\end{split}
\end{equation}
It follows also $
\M(\M(\hh^{1/2})  \hh^{1/2}) \M(\hh^{1/2})  \hh^{1/2} = \M(\hh) \hh$.
Hence, the formulations~\eqref{eq:SG_SWE_cons} and~\eqref{eq:SG_SWE_Roe} are identical.
\end{proof}
Using a stochastic basis that satisfies \eqref{eq:basis_P1}, we have seen that the SG-SW formulations considered in \cite{dai_etal_2021} and \cite{Gerster_Herty_20} are identical. However, since entropy functions are not unique, we focus also on \eqref{eq_ent_flux_dai} and \eqref{eq_ent_flux_gerster}. Here, we present the following result: 

\begin{theorem}\label{th_identical}
      For any stochastic basis satisfying~\eqref{eq:basis_P1}, the corresponding entropy-entropy flux pair in conservative variables~\eqref{eq_ent_flux_dai} is identical to the entropy-entropy flux pair  in conservative variables~\eqref{eq_ent_flux_gerster}. From this, it follows directly that both the entropy variables and entropy potentials are identical. 
\end{theorem}
\begin{proof}
To demonstrate that all quantities are identical can be seen by elementary calculations which are given in the Appendix \ref{sub_sub_proof_1}. 
\end{proof}

The eigenvalues of the stochastic Galerkin flux Jacobian~\eqref{eq:SG_Jacobian_cons} for bases satisfying~\eqref{eq:basis_P1} can be obtained as follows. Assuming~\eqref{eq:basis_P1} holds, Eq.~\eqref{eq:SG_Jacobian_cons} is identical to
\begin{equation}
\label{eq:SG_Jac_semi_diag}
\begin{pmatrix}
V & O_{K} \\
O_{K} & V
\end{pmatrix}
\begin{pmatrix}
O_{K} & I_{K}\\
g \Lambda(\hat{h}) - \Lambda^2(\hat{q})\Lambda^{-2}(\hat{h}) & 2 \Lambda(\hat{q})\Lambda^{-1}(\hat{h}) 
\end{pmatrix}
\begin{pmatrix}
V^T & O_{K} \\
O_{K} & V^T
\end{pmatrix}.
\end{equation}
The matrix~\eqref{eq:SG_Jac_semi_diag} is similar (hence have identical eigenvalues) to the block diagonal matrix where the $k$th diagonal block is the $2 \times 2$ matrix:
\[
\begin{pmatrix}
0 & 1\\
g\lambda_k(\hat{h})-\lambda_k^2(\hat{q}) \lambda_k^{-2}(\hat{h}) &
2\lambda^{-1}_k(\hat{h})\lambda_k(\hat{q})
\end{pmatrix}.
\]
Note the resemblance to the Jacobian of the deterministic SW equations; the eigenvalues are those of the deterministic Jacobian with any deterministic quantity $w$ replaced by the $k^\text{th}$ eigenvalue of the SG matrix of $\hat{w}$, i.e., $\lambda_k(\hat{w})$. The eigenvalues of~\eqref{eq:SG_Jacobian_cons} are thus given by
$
\lambda_k(\hat{u}) \pm \sqrt{g \lambda_k(\hat{h})}, \quad k=1,\dots, K.
$
As shown in~\cite{pettersson2015polynomial}, for Haar wavelets, the SG matrix eigenvalues are given by
$
\lambda_{k}(\hat{w}) = (\sqrt{K} V^T \hat{w})_{k}.
$
Hence, the spectrum can be obtained directly from
\revOne{
\begin{equation}
\label{eq:eigVals}
\sqrt{K}V^T \hat{u} \pm \sqrt{gK} (V^T\hat{h})^{1/2},
\end{equation}
}
where the power $1/2$ is to be applied entry-wise on the vector.

%% file: DG.tex
\section{Structure preserving numerical methods}\label{sec_DGmethod}

In the following section, we briefly introduce the DGSEM approach and highlight its key properties. A central aspect is the use of entropy-conservative two-point fluxes within a flux-differencing framework. To this end, we provide an overview of the numerical entropy conservative flux construction process and discuss the relationship between the fluxes developed in \cite{bender_2023} and \cite{dai2023energy}. We then extend our fluxes to two spatial dimensions and relate the resulting formulation to the recent work in \cite{epshteyn2024energy}.

\subsection{Entropy stable and well-balanced discontinuous Galerkin for the SG-SW formulation}

Over the past decade, there has been a marked development and analysis
of entropy stable nodal DGSEM applied to deterministic PDEs, e.g.,
\cite{carpenter2014entropy, chen2020review, gassner2016split,  winters2021}.
In the context of SW models special care is taken
to assemble an entropy stable high-order approximation that
remains well-balanced, e.g.,
\cite{fu2022high, gassner2016well, gaburro2023high, mantri2024fully, wintermeyer2017entropy, wu2021high, xing2014exactly}. Here, we adapt the entropy stable DGSEM framework to the SG-SW system \eqref{eq:SG_SWE_cons}. For simplicity, we focus on one spatial dimension and assume 
that the bottom topography $b$ is continuous as it simplifies the
discussion and removes the need to discuss a path-conservative
discretization of the nonconservative terms. For more details, we refer to the literature \cite{ winters2015comparison, wintermeyer2017entropy}.

We consider a one-dimensional spatial domain $\Omega=[x_L,x_R]$. Then, the construction of the entropy stable DGSEM can be summarized as follows. 
We multiply the system \eqref{eq_SW_pure} by a test function $\boldsymbol{\varphi}$ and integrate over the domain $\Omega$. Then, we divide $\Omega$ into non-overlapping elements $E_\kk$, $\kk=1, \ldots, \mathfrak{K}$. Instead of working in each element separately, all calculations are done in the reference element $E_0 = [-1,1]$ and we create a mapping between the physical coordinates $x$ and computational coordinate $\chi$ on the reference element $E_0$ via
    \begin{equation*}
    \chi(x) = 2\left(\frac{x-x_{k}}{x_{k+1}-x_{k}}\right) - 1,\quad\textrm{such that}\quad \chi\in[-1,1].
    \end{equation*}
    This introduces a Jacobian term on each element of the form
$
    \chi_x = \frac{2}{\Delta x}$ with $ \Delta x = x_{\kk+1} - x_\kk. 
$   From the mapping, we apply the chain rule to rewrite spatial derivatives from physical coordinates into the computational coordinates. The numerical solution and physical flux with nodal polynomials of degree $N$ are written in the Lagrange basis, e.g., 
    \begin{equation}
    u(x,t) \approx U(x,t) = \sum_{j=0}^NU_j(t)\ell_j(\chi),
    \end{equation}
    where the interpolation nodes are taken to be $N+1$ Legendre-Gauss-Lobatto (LGL) nodes. Now, we apply two times integration-by-parts where after the first round the discontinuity at the element boundary is resolved by using a numerical flux in the spirit of FV methods. The second integration-by-parts yield us the strong form DGSEM\footnote{ Note that the resulting flux penalty term at each element interface is analogous to the SAT terms present in provably stable finite difference approximations~\cite{gassner2013skew}.}. We now select the test function to be the Lagrange basis polynomials $\varphi = \ell_i(\chi)$ where $i=0,\ldots,N$ and approximate the integrals with LGL quadrature rules where we additionally collocate the interpolation and quadrature nodes.     This allows us to exploit the Kronecker delta property of the Lagrange basis in the name of computational efficiency~\cite{kopriva2009implementing}.
    This yields finally to the DGSEM. However, the discretization is not in general entropy stable. The key ingredient is now to replace the  collocated flux in the volume with an entropy conservative, numerical two-point flux (denoted with a $\#$ symbol) that extends a low-order entropy conservative finite volume flux to high-order~\cite{carpenter2014entropy, fisher2013, fisher2013_2} and results in a split-form discretization for the volume flux. 
The resulting semi-discrete nodal DGSEM takes the form
\begin{equation}\label{eq:semiDG}
\frac{\Delta x}{2}\dot{U}_i + 2\sum_{j=0}^N\mathcal{D}_{ij}F^{\#}(U_i,U_j) + \frac{\tau_i}{\omega_i}\left(F^* - F\right) = S_i,\quad\text{where}\quad i = 0,\ldots, N,
\end{equation}
with derivative matrix $\mathcal{D}$ approximating the first derivative $\frac{\partial}{\partial x}$, $\omega_i$ be the Gauss-Lobatto quadrature weights (which defines as well a diagonal norm mass matrix $\mathcal{M}=  \text{diag}(\omega_0,\ldots,\omega_N) $), and $\tau_i$ represent the contribution at the boundary. We have $$\tau_i=\begin{cases}
-1 \text{ for } i=0,\\
1 \; \text { for } i=N, \\
0 \;\text{ else}.
\end{cases}
$$
The nonconservative term on the right hand side of the
SW equations \eqref{eq_SW_pure} is approximated as in \cite{wintermeyer2017entropy}.
The selection of the symmetric, two-point volume flux $F^{\#}(U_i,U_j)$ and the two-point surface flux $F^*$ play a crucial role to guarantee the overall properties of the method. To point this out, it has been proven in several works \cite{oeffner2023, chen2020review,gassner2016split}, but goes originally back of an investigation by Fisher et al. \cite{fisher2013, fisher2013_2},  for a hyperbolic conservation law, i.e. the source term is equal to zero in \eqref{eq:semiDG} . It was demonstrated in such case: 
\begin{theorem}
    If $F^{\#}(U_i,U_j) $ is consistent and symmetric, then \eqref{eq:semiDG} is conservative and high-order accurate. If we further assume that $F^{\#}(U_i,U_j) $ is entropy conservative, then the semi-discrete scheme \eqref{eq:semiDG} is also entropy conservative within a single element. In addition, if the numerical flux $F^*$ at the element interface is entropy conservative (dissipative), the the scheme is entropy conservative (dissipative). 
\end{theorem}
This is mainly a result for hyperbolic conservation laws but can be extended to balance laws including the source discretization. In terms of the deterministic SW, we refer to \cite{winters2015comparison, ranocha2017, wintermeyer2017entropy}. 
In \cite{bender_2023}, the approach has been used to solve the SG-SW equations~\eqref{eq:SG_SWE_Roe} where specific symmetric, entropy-conservative, two-point fluxes $F^{\#}$ have been constructed. Meanwhile in \cite{dai2023energy}, the FV framework has been applied where also adequate entropy-conservative, two-point fluxes $F^{\#}$ have to be constructed and used.  Here, we point out that even when the stochastic basis satisfies~\eqref{eq:basis_P1} and both the SG-SW equations and the entropy-entropy flux pairs are identical, the numerical fluxes do not have to be identical, as we will see and focus on in subsection~\ref{subsec_two_point_Dai}.

\subsection{Entropy conservative numerical two-point fluxes} \label{subsec_two_point_Dai}

Even if the underlying methods in \cite{bender_2023} and \cite{dai2023energy} are different, the construction process for entropy conservative two-point fluxes is always the same and based on the general idea of Tadmor \cite{Tadmor_2003entropy}.
To describe it, we first introduce  the average and jumps between two degrees of freedom  $i$ and $i+1$ which can be for instance the values at the cell interfaces from one to the other element or quadrature/interpolation points in DGSEM:
$$
\avg{a}_{i+1/2} := \frac{1}{2} \left(a_i+a_{i+1} \right) \text{ and } \jump{a}_{i+1/2}:= a_{i+1}-a_i.
$$
Following \cite{Tadmor_2003entropy}, an entropy conservative numerical flux  $F^{\#,con}$ satisfies\footnote{We introduce the superscript $con$ to denote the conservative part meaning the case when the source term is equal to zero. } 
\begin{equation}\label{eq_Tadmor}
\jump{\hat{w}}\cdot F^{\#,con} = \jump{\Phi}, 
\end{equation}  
where  $\jump{\Phi}$ and $\jump{\hat{w}}$ denote the jump in the  entropy potential  and entropy variables. In \eqref{eq_Tadmor}, the source has been neglected; we explain later how the source has to be incorporated.
In \cite{dai2023energy}, the authors used conservative variables and derived a numerical flux with corresponding source term given by 
\begin{equation}\label{eq_conservative_flux_dai}
\begin{aligned}
F^{\#, con}_{i+1/2}= \begin{pmatrix}
    \M\left( \overline{\hh}_{i+1/2} \right) \overline{\hv}_{i+1/2} \\
    \frac{g}{2} \overline{\left(\M (\hh) \hh\right)}_{i+1/2} +
    \M(\overline{\hv}_{i+1/2}) \M(\overline{\hh}_{i+1/2}) \overline{\hv}_{i+1/2}
    \end{pmatrix}
    \\
S^{\#}_{i+1/2}= 
\begin{pmatrix}
   0 \\
    -\frac{g}{2\Delta x} \left(\M (\overline{\hh}_{i+1/2}) \revTwo{\jump{\hb}_{i+1/2}}   + \M (\overline{\hh}_{i-1/2}) \revTwo{\jump{\hb}_{i-1/2}} \right). 
\end{pmatrix}
\end{aligned}
\end{equation}
A straightforward comparison between the flux \eqref{eq_conservative_flux_dai} and the results in \cite{bender_2023} is not possible since the representation in~\cite{bender_2023} used entropy variables whereas~\cite{dai2023energy} applied conservative variables. Furthermore, the focus in~\cite{bender_2023} was the case without bottom topography (in the appendix there is a short outlook incorporating bottom topography) whereas in~\eqref{eq_ent_flux_dai} the source discretization has already been included in the main discretization. To highlight the difference from~\cite{bender_2023} and~\cite{dai2023energy}, we briefly outline the general procedure for the entropy construction with source term:
\begin{enumerate}
\item Derive entropy conservative numerical fluxes $F_{i,k}^{\#,con}$ in matrix-vector notations using \eqref{eq_Tadmor} with zero bottom topography. Split the jumps in the potential and entropy variables in two parts and collect the parts corresponding to first and second entropy variable  $\jump{\hat{w}_{1/2}}$ together with the corresponding numerical fluxes. By direct variable comparison, we can derive entropy conservative fluxes.
\item Use the direct transformation between entropy-primitive variables and conservative variables and the properties of $\M$ to obtain an expression of the numerical entropy conservative flux in conservative variables similar to \eqref{eq_conservative_flux_dai}.
\item To include the bottom topography, we follow the approach in~\cite{zbMATH02174321} and investigate the numerical flux 
\begin{equation}\label{eq_extended}
F_{i,k}^{\#,ext}=F_{i,k}^{\#,con}+S^\#_{i,k}
\end{equation}
where $S_{i,k}$ denotes the source term between two degrees of freedom $i$ and $k$.To include bottom topography, the entropy variable $\hw_1$ from~\eqref{eq:entropy_variable_Dai} is extended by the jump in the bottom topography $\jump{b}$ resulting in the identity  
\begin{equation}\label{eq_Tadmor_2}
\jump{\hw}\cdot F^{\#} = \jump{\Phi}+g F_1^\# \jump{\hat{b}}.
\end{equation}  
Thus,  to balance the rate of change of the entropy, the discretization of the source term has to fulfill
\revTwo{
\begin{equation}\label{eq_calculated}
g F_1^\# \cdot\jump{\hb}_{i,k}+ \hw_{k} \cdot S^\#_{k,i}-\hw_{i} \cdot S^\#_{i,k}\ =0.
\end{equation}
}
\end{enumerate}
Following this recipe, we can derive the following result. 
\begin{theorem}\label{th_entropy_conservative}
Considering the SG-SW \eqref{eq:SG_SWE_Roe} with a  basis that satisfies \revTwo{\eqref{eq:basis_P1}, the }entropy conservative numerical flux with bottom topography included is given via 
\begin{equation}\label{eq_flux_conservative}
\begin{aligned}
F^{\#}_{1} &= \frac{1}{2g} \left(
\frac{1}{2}\avg{\M^2(\hv)\hv} + g\avg{\M(\hv)\hh} -\frac{1}{2}\avg{\M^2(\hv)} \ \avg{\hv} + g\M(\avg{\hv}) \avg{\hv}
\right)\\[0.1cm]
F^{\#}_{2} &= \frac{g}{2} \avg{\M(\hh)\hh}  + \M^2(\avg{\hv})\avg{\hh} +
\Bigg(- \frac{1}{2} \avg{\M(\hh)\M(\hv)\hv} 
- \frac{1}{4g} \avg{\M^3(\hv)\hv} -\frac{1}{2g}  \M^2(\avg{\hv}) \avg{\hv^2} 
 + \frac{1}{2} \avg{\M^2(\hv)}  \avg{\hh}-  \frac{1}{4g} \avg{\M^2(\hv)}  \avg{\hv^2}  \\
 &
-\frac{3}{8g}\left( 
 \avg{\M^3(\hv)} \ \avg{\hv}
 + \M(\avg{\hv}) \avg{\M^2(\hv)} \ \avg{\hv} + 2 \M^2(\avg{\hv})\M(\avg{\hv})\ \avg{\hv}\right) 
 + \frac{1}{2g} \avg{\M^2(\hv)\hv} + \frac{1}{2g} \avg{\M^2(\hv)} \ \avg{\hv} + \frac{1}{g} \M^2(\avg{\hv}) \avg{\hv} \Bigg) \\
\end{aligned},
\end{equation}
with corresponding source discretization
\begin{equation}\label{eq:bterm}
S^{\#}_{i,k} = \begin{pmatrix}
0 \\
\frac{1}{2} g \left(\M (\overline{\hh}_{i,k}) \jump{\hb}_{i,k} \right)+  \frac{1}{8g}  \left( (\M(\hv) )_k- (\M(\hv) )_i  \right)^2 \jump{\hat{b}}_{i,k}
\end{pmatrix},
\end{equation}
where  we use the notation $\hv^2 \equiv \M(\hv) \hv$ in \eqref{eq_flux_conservative}.

\end{theorem}
\begin{proof}
The proof follows the recipe and steps written above. It is based on simple, but lengthy calculations. For completeness,  it can be found in the Appendix \ref{sub_ap_2}. 
Note that in~\eqref{eq_flux_conservative}, we drop the notation of the corresponding degree of freedom $i,k$ of the averages  to shorten the presentation. 
\end{proof}

\subsubsection*{Discussion about the relations}

Clearly, the entropy-conservative numerical fluxes \eqref{eq_conservative_flux_dai} and \eqref{eq_flux_conservative} are not the same. This is evident because \eqref{eq_flux_conservative} involves higher-order (third-order and higher) average terms concerning velocity, whereas the flux in~\eqref{eq_conservative_flux_dai} contains  third (second) order terms. In $F_1^\#$, the first term differs also. In $F_2^\#$, the flux \eqref{eq_conservative_flux_dai} is incorporated in the first two terms of \eqref{eq_flux_conservative}, with additional terms arising from the averaging process. We can draw a parallel to the deterministic case as discussed in \cite{ranocha2017}, which elaborates on generating two-parameter families of entropy-conservative fluxes. The deterministic counterparts of \eqref{eq_flux_conservative} and \eqref{eq_conservative_flux_dai} fit within this two-parameter family with specific parameter choices. The flux \eqref{eq_flux_conservative} corresponds to both parameters set to one, while \eqref{eq_conservative_flux_dai} corresponds to one parameter being adjusted to eliminate all higher-order terms concerning the conservative variables and the other is set to one. These types of fluxes have been constructed and studied in \cite{winters2015comparison, wintermeyer2017entropy, fjordholm2012arbitrarily}. Their main advantage lies in their simpler representation. However, as also explored in \cite{ranocha2017}, no significant advantages in terms of 
accuracy and stability have been observed between use of different entropy-conservative fluxes.

Finally, note that the source discretizations also have to be adapted for the higher-order average processes of the velocity components. 
However, note that the 
Dai et al. flux \eqref{eq_ent_flux_dai} is well-balanced for lake at rest \eqref{eq_well_balanced_SW} as pointed out in \cite{dai2023energy}.
The first terms in $F_2^\#$ of \eqref{eq_flux_conservative} with \eqref{eq:bterm} are identical to \eqref{eq_well_balanced_SW} whereas only higher-order velocity components are added which are zero for lake at rest \eqref{eq_well_balanced_SW}. As a result, we obtain that the flux \eqref{eq_flux_conservative} with source term \eqref{eq:bterm} applied in a DG framework yields a scheme which is entropy conservative and discretely well-balanced for lake at rest.

We have pointed out the similarities of the results between \cite{dai2023energy} and \cite{bender_2023}. Both approaches yield  identical results in terms of entropy function, fluxes, and potential if a \eqref{eq:basis_P1} satisfying basis is used. The only difference are the  numerical fluxes 
due to different average process. 
Next, we will extend the investigation to two space dimensions. Here, we focus on the approach of \cite{dai2023energy, epshteyn2024energy} for simplicity, working with the average process in conservative variables.

\subsection{The SG-SW formulation in two space dimensions}

We have the following two-space dimension SG-SW formulation 
\begin{equation}
\label{eq:SG_SWE_cons_2d} 
\frac{\partial}{\partial t}
\begin{pmatrix}
\hh\\
\hq_{\mathfrak{1}}\\
\hq_{\mathfrak{2}}
\end{pmatrix}
+
\frac{\partial}{\partial x}
\begin{pmatrix}
\hq_{\mathfrak{1}}\\
\M(\hq_{\mathfrak{1}}) \M^{-1}(\hh)  \hq_{\mathfrak{1}} 
+ \frac{1}{2}g  \M(\hh)  \hh\\
\M(\hq_{\mathfrak{2}}) \hv_{\mathfrak{1}}
\end{pmatrix}
+
\frac{\partial}{\partial y}
\begin{pmatrix}
\hq_{\mathfrak{2}}\\
\M(\hq_{\mathfrak{1}}) \hv_{\mathfrak{2}}
\\
\M(\hq_{\mathfrak{2}}) \M^{-1}(\hh)  \hq_{\mathfrak{2}} 
+ \frac{1}{2}g  \M(\hh)  \hh
\end{pmatrix}
=
\begin{pmatrix}
\hat{0} \\
-g \M(\hh) \frac{\partial}{\partial x} \hat{b}\\
-g \M(\hh) \frac{\partial}{\partial y} \hat{b}
\end{pmatrix},
\end{equation}
with $\hv_{\mathfrak{1}}=\M^{-1}(\hh) \hq_{\mathfrak{1}}$ and 
$\hv_{\mathfrak{2}}=\M^{-1}(\hh) \hq_{\mathfrak{2}}$ where $\bullet_{\mathfrak{1}/\mathfrak{2}}$ denote the quantity, e.g. velocities or discharge, in $x$ or $y$ directions.
Note the form of the moment equations in \eqref{eq:SG_SWE_cons_2d} differs from that presented in \cite{epshteyn2024energy} but, importantly, it has the same symmetries meaning the positions of $\hv_{\mathfrak{\cdot}}$. 
A straightforward extension from the works \cite{dai2023energy, winters2015comparison} and analogously to \cite{epshteyn2024energy} yields the following result: 
\begin{theorem}\label{th_2d}
For any stochastic basis satisfying~\eqref{eq:basis_P1}, the SG-SW equations \eqref{eq:SG_SWE_cons_2d} in two dimensions is equipped with the following entropy - entropy flux pair: 
\begin{equation}\label{eq_ent_flux_dai_2D}
\begin{aligned}
\hat{\eta}(\hh,\hq_{\mathfrak{1}}, \hq_{\mathfrak{2}}) &= \frac{1}{2}\left( \hq_{\mathfrak{1}}^{T}\M^{-1}(\hh)\hq_{\mathfrak{1}} + \hq_{\mathfrak{2}}^{T}\M^{-1}(\hh)\hq_{\mathfrak{2}}+ g \hh^{T}\hh\right) + g\hh^{T}\hat{b},\\
\hat{H}_{\mathfrak{1}}(\hh,\hq_{\mathfrak{1}}, \hq_{\mathfrak{2}}) &= \frac{1}{2} \hv_{\mathfrak{1}}^{T}\M(\hq_{\mathfrak{1}}) \hv_{\mathfrak{1}} + \frac{1}{2} \hv_{\mathfrak{2}}^{T}\M(\hq_{\mathfrak{1}}) \hv_{\mathfrak{2}} + g\hq_{\mathfrak{1}}^{T}\hh + g\hq_{\mathfrak{1}}^{T}\hat{b},\\
\hat{H}_{\mathfrak{2}}(\hh,\hq_{\mathfrak{1}}, \hq_{\mathfrak{2}}) &= \frac{1}{2} \hv_{\mathfrak{1}}^{T}\M(\hq_{\mathfrak{2}}) \hv_{\mathfrak{1}} +\frac{1}{2} \hv_{\mathfrak{2}}^{T}\M(\hq_{\mathfrak{2}}) \hv_{\mathfrak{2}}  + g\hq_{\mathfrak{2}}^{T}\hh + g\hq_{\mathfrak{2}}^{T}\hat{b},  
\end{aligned}
\end{equation} 
\revTwo{where} the entropy variables and potential are given via
\begin{equation}\label{eq:entropy_variable_Dai_2D}
\hw:= \left(\frac{\partial \hat{\eta}}{\partial \hu} \right)^T= \left( -\frac{1}{2}\hv_{\mathfrak{1}}^T\M(\hv_{\mathfrak{1}})- \frac{1}{2}\hv_{\mathfrak{2}}^T\M(\hv_{\mathfrak{2}})+g (\hh+\hat{b})^T, \hv_{\mathfrak{1}}, \hv_{\mathfrak{2}}\right)^T,
\end{equation}
and 
\begin{equation}\label{eq_potential_2D}
\Phi= \hw \hat{F}- (\hat{H}_{\mathfrak{1}}+\hat{H}_{\mathfrak{2}}) = \frac{1}{2} g \hv_{\mathfrak{1}}^T \M( \hh) \hh +\frac{1}{2} g \hv_{\mathfrak{2}}^T \M( \hh) \hh.
\end{equation}
Consider a tensor structure grid. Then, an entropy conservative flux is given via 
\begin{equation}\label{eq_conservative_flux_2d}
\begin{aligned}
F^{\#}_{\mathfrak{1}, i+1/2}= \begin{pmatrix}
    \M\left( \overline{\hh}_{i+1/2} \right) \overline{\hv}_{\mathfrak{1}, i+1/2} \\
    \frac{g}{2} \overline{\left(\M (\hh) \hh\right)}_{i+1/2} +
    \M(\overline{\hv}_{\mathfrak{1},i+1/2}) \M(\overline{\hh}_{i+1/2}) \overline{\hv}_{\mathfrak{1},i+1/2}\\
      \M(\overline{\hv}_{\mathfrak{1},i+1/2}) \M(\overline{\hh}_{i+1/2}) \overline{\hv}_{\mathfrak{2},i+1/2}
    \end{pmatrix},
    \\
    F^{\#}_{\mathfrak{2}, i+1/2}= \begin{pmatrix}
    \M\left( \overline{\hh}_{i+1/2} \right) \overline{\hv}_{\mathfrak{2}, i+1/2} \\
      \M(\overline{\hv}_{\mathfrak{2},i+1/2}) \M(\overline{\hh}_{i+1/2}) \overline{\hv}_{\mathfrak{1},i+1/2} \\
       \frac{g}{2} \overline{\left(\M (\hh) \hh\right)}_{i+1/2} +
    \M(\overline{\hv}_{\mathfrak{2},i+1/2}) \M(\overline{\hh}_{i+1/2}) \overline{\hv}_{\mathfrak{2},i+1/2}\\
    \end{pmatrix},
\end{aligned}
\end{equation}
with source discretization
 \begin{equation}\label{eq_conservative_source_2d}
\begin{aligned}   
S^{\#}_{i+1/2}= 
\begin{pmatrix}
   0 \\
    -\frac{g}{2\Delta x} \left(\M (\overline{\hh}_{i+1/2}) \jump{\hat B}_{i+1/2}   + \M (\overline{\hh}_{i-1/2}) \jump{\hat B}_{i-1/2} \right)\\
        -\frac{g}{2\Delta y} \left(\M (\overline{\hh}_{i+1/2}) \jump{\hat B}_{i+1/2}   + \M (\overline{\hh}_{i-1/2}) \jump{\hat B}_{i-1/2} \right)
\end{pmatrix},
\end{aligned}
\end{equation}
resulting in a well-balanced scheme for lake-at-rest. 
\end{theorem}
\begin{proof}[Sketch of the proof]
The calculation of the entropy variable \eqref{eq_ent_flux_dai_2D} and potential \eqref{eq_potential_2D} is  straightforward, whereas the numerical fluxes can be derived using the approach presented in Subsection~\ref{subsec_two_point_Dai}.
 To ensure that we have an entropy entropy flux pair, we have to demonstrate that 
the entropy function is a convex function in the conservative variables, e.g. showing that the entropy Hessian is strictly positive definite, and that the compatibility condition $
\nabla_{\hu } \hat{\eta} \cdot \nabla_{\hu} \hat{F} = \nabla_{\hu} \hat{H},
$
is satisfied. 
The proof is a straightforward extension from the one space dimensional case, and can be found in~\cite{epshteyn2024energy}.
\end{proof}

%% file: Numerical_results.tex
\section{Numerical results}\label{se_numerics}

We investigate the entropic properties and well-balancedness of the stochastic Galerkin discretization in one and two spatial dimensions.
\revTwo{Throughout the numerical results we denote the use of entropy conservative (EC) or entropy stable (ES) fluxes.}
The spatial discretizations for the split form DGSEM are available in Trixi.jl
\cite{ranocha2022adaptive,schlottkelakemper2021purely}.
For time integration we use the five-stage, four-order explicit Runge-Kutta method of Carpenter and Kennedy \cite{Carpenter&Kennedy:1994} implemented in
OrdinaryDiffEq.jl \cite{rackauckas2017differentialequations}.
We use Plots.jl \cite{christ2023plots} to visualize the results.
All source code needed to reproduce the numerical experiments is
available online in our reproducibility repository~\cite{oeffner2026numericalRepro}.

\subsection{One dimensional problems}

For different numbers of Haar wavelets and polynomial orders we verify the well-balanced property of the one-dimensional approximation in Section~\ref{sec:wb1D} and the entropy conservation / stability of the stochastic DG scheme in Section~\ref{sec:eces1D}.

\subsubsection{Well-balancing}\label{sec:wb1D}

These tests all use polynomials of degree $N=3$ in each element on a 16 element mesh of the domain $[0, 20]$. 
The final time is set to $t_{\text{final}} = 100$. 
The boundary conditions are periodic and we use the EC flux and nonconservative terms \eqref{eq_conservative_flux_dai}.
We also ran well-balancing tests using the ES flux at element interfaces and obtained similar well-balancedness errors.

We use both the uncertain height and uncertain position one dimension bottom topography $b(x, \xi)$ described in \ref{sec:uncertain_bottom1D}. 
The initial conditions for the well-balancing tests are
\[
\begin{aligned}
\mathbf{H}(x,t=0) &= \left\{
\begin{array}{ll}
\frac{4}{3}, & \mbox{first wavelet}\\
0, & \mbox{otherwise},
\end{array}
\right.
\qquad
\mathbf{v}(x,t=0) = 0.
\end{aligned}
\]
So, the uncertainty in the bottom topography $b(x, \xi)$ propagates into the primitive variables $\mathbf{h} = \mathbf{H} - b$.
First, we test a one dimensional uncertain position and present the discrete $L_1$ well balanced error with two, four, and eight Haar wavelets.
Integrating up to a final time the well-balancedness error in each water height is reported in Table~\ref{tab:1Dheight}
\begin{table}[H]
\centering
\caption{1D uncertain height, final time $100$, polynomial degree $3$, time step $0.1$, \revTwo{EC fluxes in the volume and on the surface}.}
\label{tab:1Dheight}
\begin{tabular}{c S[table-format=1.2e-2]
S[table-format=1.2e-2]
S[table-format=1.2e-2]
S[table-format=1.2e-2]
S[table-format=1.2e-2]
S[table-format=1.2e-2]
S[table-format=1.2e-2]
S[table-format=1.2e-2]}
\toprule
{$K$}
& {$H_1$} & {$H_2$} & {$H_3$} & {$H_4$}
& {$H_5$} & {$H_6$} & {$H_7$} & {$H_8$} \\
\midrule
2
& 3.23005511e-15
& 1.11952624e-15
& \multicolumn{1}{c}{---} & \multicolumn{1}{c}{---} & \multicolumn{1}{c}{---} & \multicolumn{1}{c}{---} & \multicolumn{1}{c}{---} & \multicolumn{1}{c}{---}
\\
4
& 3.67645729e-15
& 8.03321360e-16
& 1.02575354e-15
& 1.92272206e-15
& \multicolumn{1}{c}{---} & \multicolumn{1}{c}{---} & \multicolumn{1}{c}{---} & \multicolumn{1}{c}{---}
\\
8
& 2.41126563e-15
& 5.61454923e-16
& 8.30676292e-16
& 1.07638418e-15
& 3.94810437e-16
& 6.16995969e-16
& 6.39237973e-16
& 6.66624510e-16
\\
\bottomrule
\end{tabular}
\end{table}

We consider the same initial condition as above only now with a bottom topography $b(x, \xi)$ with uncertain position.
The discrete $L_1$ well balanced error at time $t_{\text{end}}$ with two, four, and eight Haar wavelets is presented in Table~\ref{tab:1Dposition}.
\begin{table}[H]
\centering
\caption{1D uncertain position, final time $100$, polynomial degree $3$, time step $0.1$, \revTwo{EC fluxes in the volume and on the surface}.}
\label{tab:1Dposition}
\begin{tabular}{c S[table-format=1.2e-2]
S[table-format=1.2e-2]
S[table-format=1.2e-2]
S[table-format=1.2e-2]
S[table-format=1.2e-2]
S[table-format=1.2e-2]
S[table-format=1.2e-2]
S[table-format=1.2e-2]}
\toprule
{$K$}
& {$H_1$} & {$H_2$} & {$H_3$} & {$H_4$}
& {$H_5$} & {$H_6$} & {$H_7$} & {$H_8$} \\
\midrule
2
& 5.04226290e-15
& 4.57821114e-14
& \multicolumn{1}{c}{---} & \multicolumn{1}{c}{---} & \multicolumn{1}{c}{---} & \multicolumn{1}{c}{---} & \multicolumn{1}{c}{---} & \multicolumn{1}{c}{---}
\\
4
& 4.81443589e-15
& 3.83481301e-14
& 4.26846690e-15
& 4.53824730e-15
& \multicolumn{1}{c}{---} & \multicolumn{1}{c}{---} & \multicolumn{1}{c}{---} & \multicolumn{1}{c}{---}
\\
8
& 2.50725366e-15
& 1.98419593e-14
& 9.61814156e-15
& 1.09490881e-14
& 1.91969691e-14
& 3.65771211e-14
& 3.79477876e-14
& 1.44905223e-14
\\
\bottomrule
\end{tabular}
\end{table}
\revOne{Further, for the test case with an uncertain position for $b(x,\xi)$ we purposefully create a scheme that is not well-balanced to highlight its importance. For this, we use the standard collocated DG flux discretization 
\begin{equation}
    \label{eq:standardFlux}
\mathbf{F}^{standard}(U) = \begin{pmatrix}
\hq\\
\M(\hq) \M^{-1}(\hh)  \hq 
+ \frac{1}{2}g  \M(\hh)  \hh
\end{pmatrix}    
\end{equation}
in the volume and on the surface.
Here, as in the deterministic case, the flux \eqref{eq:standardFlux} does not produce a provably well-balanced discretization \cite{winters2015comparison}.
Thus, the stochastic DG method will produce spurious waves on the order of the mesh size that pollute the solution and may lead to instability, e.g. due to negative water heights.
We use four wavelets $K=4$, polynomials of degree $3$, and a final time of $t=100$.
Results in Table~\ref{tab:nonWB} demonstrate such behavior where a coarse mesh may be unstable or have significant pollution of the $L_1$ well balanced error but, upon mesh refinement, these spurious errors are significantly reduced.}
\begin{table}[H]
\centering
\revOne{
\caption{\revOne{1D uncertain height, final time $100$, polynomial degree $3$, standard DG fluxes fluxes in the volume and on the surface.}}
\label{tab:nonWB}
\begin{tabular}{c l l
S[table-format=1.2e-2]
S[table-format=1.2e-2]
S[table-format=1.2e-2]
S[table-format=1.2e-2]}
\toprule
{$K$}
& {\# elements} & {$\Delta t$} & {$H_1$} & {$H_2$}
& {$H_3$} & {$H_4$} \\
\midrule
2
& 8
& 0.1
& \multicolumn{4}{c}{Crashes at time $\approx 78.8$}
\\
4
& 8
& 0.005
& 1.25051794e-04
& 1.50968495e-04
& 1.10686972e-04
& 1.20141417e-04
\\
8
& 16
& 0.1
& 8.01086049e-13
& 8.20768278e-13
& 5.14602480e-13
& 6.85845040e-13
\\
\bottomrule
\end{tabular}
}
\end{table}

\subsubsection{Entropy conservation and stability}\label{sec:eces1D}
To assess the entropy properties of the one-dimensional approximation we consider the domain $\Omega = [-1,1]$.
For the flow variables we use a deterministic initial condition where the initial mean value of the water height contains a discontinuity.
\[
\begin{aligned}
{H}_1(x,t=0) &= \left\{
\begin{array}{ll}
1, & \mbox{if } x\leq 0\\
0.5, & \mbox{if } x >0
\end{array}
\right.
\qquad
{H}_{2:K}(x,t=0) = 0,
\qquad
\mathbf{v}(x,t=0) = 0.
\end{aligned}
\]
together with a stochastic bottom topography with uncertain height
\[
b(x, \xi) =
\left\{
\begin{array}{ll}
\frac{1}{9}(\cos(5\pi x) + 2 + \xi)^2 & \mbox{if } |x| \leq 0.2\\[0.1cm]
\frac{1}{9}(1+\xi) & \mbox{otherwise}
\end{array}
\right.
\]
The setup is adapted from \cite{dai2023energy}.

First, we run the EC test with periodic boundary conditions.
For this test, we use polynomials of degree three, $N=3$ and 16 elements in the domain.
The EC flux and nonconservative terms from \eqref{eq_conservative_flux_dai} are used at all interior interfaces, so that the numerical approximation is nearly dissipation-free.
We run to a final time of $t_{\text{final}} = 0.65$.
Table~\ref{tab:ec1D} reports the entropy rate at the final time for increasing number of wavelets.
\begin{table}
	\centering
	\caption{Comparison of the integrated entropy rate at the final time $t_{\text{final}} = 0.65$ for the one-dimensional dam break test with periodic boundary conditions.
    All results used $N=3$ on a $16$ element mesh.}
	\label{tab:ec1D}
	\begin{tabular}{l S[table-format=1.2e-2]}
		\toprule
		\# Haar wavelets & \multicolumn{1}{c}{$\frac{1}{|\Omega|}\int S_t \, d\Omega$}\\
		\midrule
		$K=2$ & 4.01154804e-17 \\
		$K=4$ & 1.10154941e-16 \\
		$K=8$ & 6.34258271e-18 \\
		\bottomrule
	\end{tabular}
\end{table}

Next, we run a similar configuration that is entropy stable by adding local Lax-Friedrichs (LLF) type 
dissipation to the baseline EC flux.
\revOne{
The form of this numerical flux used throughout the numerical tests is
\begin{equation}
    \label{eq:LLF}
    F^*(U_L,U_R) = F^{\#,con}(U_L,U_R) - \frac{\lambda_{\max}}{2}(U_R - U_L),
\end{equation}
where the central part, $F^{\#,con}(U_L,U_R)$ from \eqref{eq_conservative_flux_dai} in 1D or \eqref{eq_conservative_flux_2d} in 2D, the left/right solution vectors $U_{L,R}$ and $\lambda_{\max}$ is the maximum eigenvalue computed at each solution state according to \eqref{eq:eigVals}.}
Here, wall boundary conditions are considered, which are neutrally stable.
We increase the resolution with $N=4$ and 64 elements and run to the same final time.
We also provide the evolution of the entropy over time in Figure~\ref{fig:es1D}.
We see that entropy decays module the effect of the outflow boundary conditions.
\begin{figure}[!ht]
    \centering 
\includegraphics[width=0.49\textwidth]{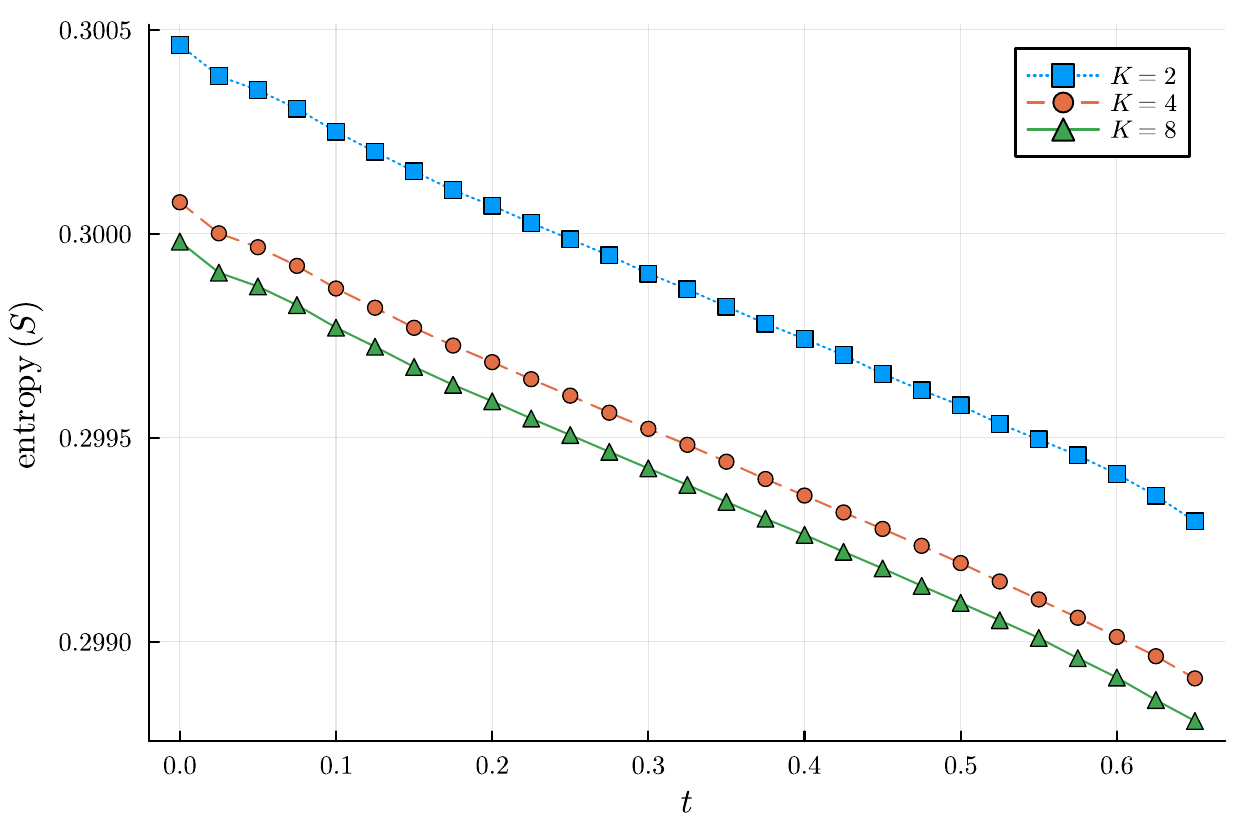}
\caption{Polynomials of degree $N=4$, $64$ elements. Evolution of the entropy for different number of wavelets using LLF type dissipation.}
\label{fig:es1D}
\end{figure}

Figure~\ref{fig:damBreak1D} shows the initial condition as well as the solution at $t_{\text{final}} = 0.65$ computed with $K=8$ Haar wavelets.
We observe that the stochastic bottom topography induces a water height
with uncertainties at later times that develops a leftward-going rarefaction wave and a rightward-going shock.
Though the ES method introduces dissipation, the solution is not overshoot-free.
\begin{figure}[!ht]
    \centering 
\includegraphics[width=0.49\textwidth]{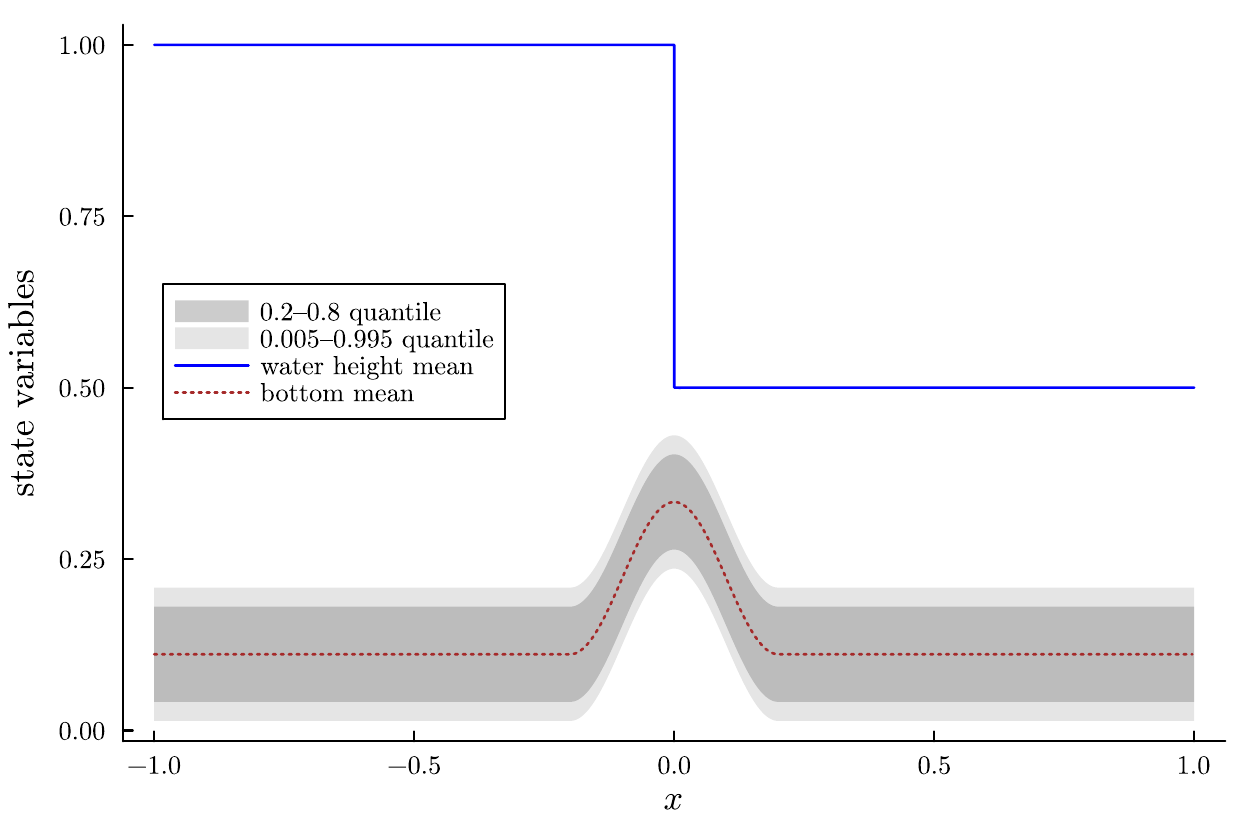}
\hfill
\includegraphics[width=0.49\textwidth]{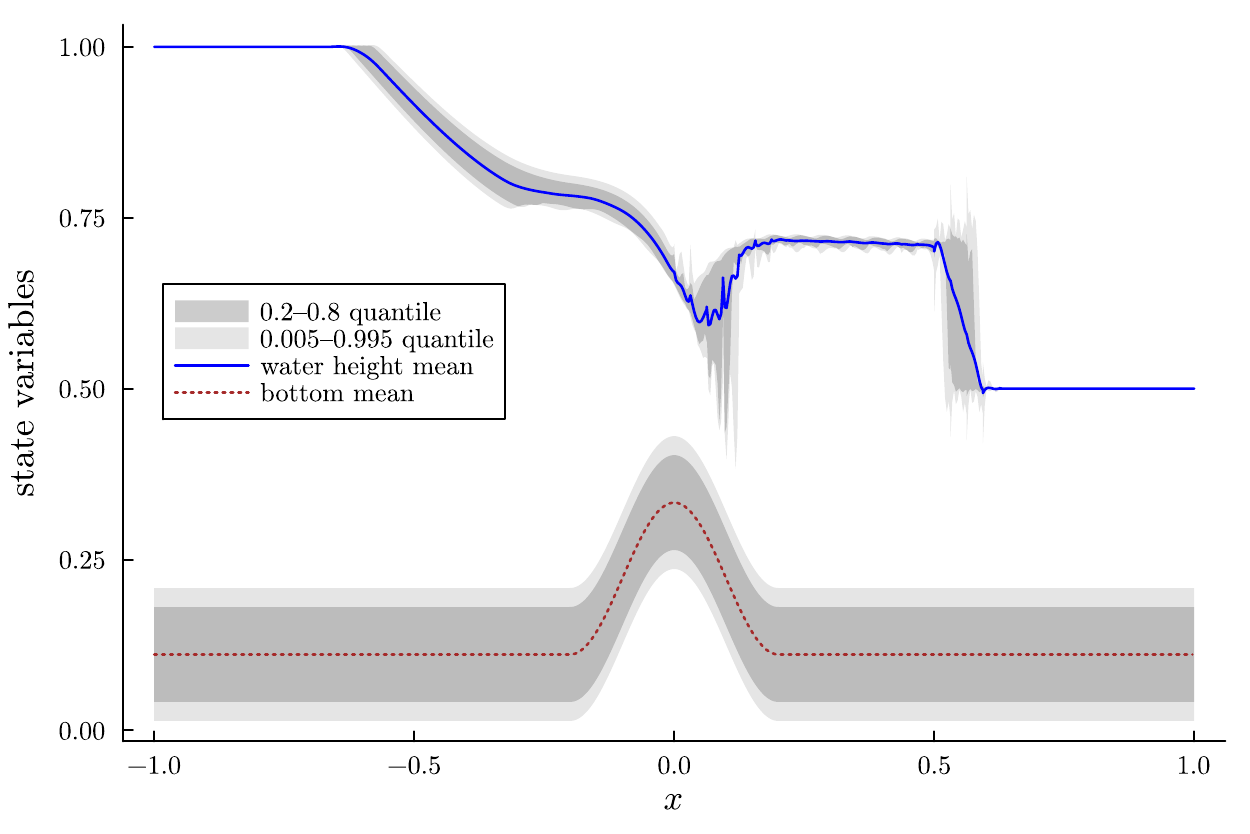}
\caption{Polynomials of degree $N=4$, $64$ elements, and $8$ wavelets. (left) Initial condition. (right) Mean solution / bottom topography and quantiles at the final time $t_{\text{final}} = 0.65$}
\label{fig:damBreak1D}
\end{figure}

\subsection{Two-dimensional problems}

We verify the high-order convergence order of the stochastic DG method with the method of manufactured solutions in Section~\ref{sec:conv2D}.
Next, in Section~\ref{sec:wb2D}, we demonstrate the well-balancedness property for different numbers of Haar wavelets.
Finally, we validate the entropy conservation and stability properties of the approximation with different dam break configurations in Section~\ref{sec:eces2D}.

\subsubsection{Convergence test}\label{sec:conv2D}

To examine the high-order spatial accuracy of the implementation we use the method of manufactured solutions.
For this, we consider two wavelets ($K=2$) as this is sufficient to exercise all terms present in \eqref{eq:SG_SWE_cons_2d}.
The domain is taken to be $[0,1]^2$ with periodic boundary conditions.
We integrate to a final time $t_{\text{final}} = 0.5$ and select a small time step $\Delta t = 5 \times 10^{-4}$ to ensure that the spatial errors dominate the approximation.

The manufactured solution \revOne{is chosen as the} initial conditions for the water height and velocities are
\begin{equation}
\label{eq:MMS2D}
\begin{aligned}
\mathbf{H}(x,y,t) &= 
\begin{pmatrix}
    1 + 0.05 \cos(2\pi x) \cos(2\pi y) \cos(2\pi t)\\[0.05cm]
    0.05 + 0.05 \sin(\pi x) \sin(\pi y) \cos(2\pi t)
\end{pmatrix},
\quad
\mathbf{v}_{\mathfrak{1}}(x,y,t)
&= 
\begin{pmatrix}
   0.64 \\[0.05cm] 0 
\end{pmatrix},
\quad
\mathbf{v}_{\mathfrak{2}}(x,y, t)
&= 
\begin{pmatrix}
    -0.75 \\[0.05cm] 0
\end{pmatrix}.
\end{aligned}
\end{equation}
\revOne{Additionally, we choose t}he bottom topography for the manufactured solution is also taken to have two modes
\begin{equation}
\label{eq:MMS2D-bottom}
\mathbf{b}(x, y) = 
\begin{pmatrix}
    0.7 + 0.05 \sin(2\pi x) \sin(2\pi y)\\[0.05cm]
    0.15 + 0.05 \cos(2\pi x) \cos(2\pi y)
\end{pmatrix}.
\end{equation}
\revOne{From this chosen solution ansatz w}e compute the conservative variable vector of the quantities in space and time where $\mathbf{h} = \mathbf{H} - \mathbf{b}$ and
\begin{equation}
\label{eq:consVec2DMMS}
\mathbf{u} = \left(\mathbf{h}, \mathbf{h}\mathbf{v}_{\mathfrak{1}},\mathbf{h}\mathbf{v}_{\mathfrak{2}}\right)^T.
\end{equation}

\revOne{The manufactured solution in \eqref{eq:consVec2DMMS} and bottom topography \eqref{eq:MMS2D-bottom} do not exactly solve \eqref{eq:SG_SWE_cons_2d}. However, it}
is straightforward to compute the temporal and spatial derivatives of the manufactured solution terms in \eqref{eq:consVec2DMMS} and \eqref{eq:MMS2D-bottom}.
But, to find expressions for the coupling of all these derivatives in the stochastic DG approximation and create an appropriate source term for the manufactured solution is complex and untractable to do explicitly.
As such, we approximate the directional derivative of each flux Jacobian term by analytically computing the derivatives of the conservative variables and evaluating the conservative fluxes from \eqref{eq:SG_SWE_cons_2d}.
For instance, to approximate the Jacobian in the $x$-direction we take
\begin{equation}
   \left(\mathbf{f}_{\mathfrak{1}}\right)_x 
   =
   \mathbf{J}_{\mathfrak{1}}
   \approx
   \frac{\mathbf{f}_{\mathfrak{1}}\left(\mathbf{u} + \varepsilon\mathbf{u}_x\right) - \mathbf{f}_{\mathfrak{1}}\left(\mathbf{u} - \varepsilon\mathbf{u}_x\right)}{2\varepsilon},
\end{equation}
where $\varepsilon = 10^{-8}$.
Then we create the source term of the manufactured solution for this complex test case to be
\[
\mathbf{S}
=
\mathbf{u}_t + \mathbf{J}_{\mathfrak{1}} \mathbf{u}_x + \mathbf{J}_{\mathfrak{2}} \mathbf{u}_y +
\begin{pmatrix}
    \mathbf{0}\\
    g \M(\mathbf{h}) \frac{\partial}{\partial x} \mathbf{b}\\
    g \M(\mathbf{h}) \frac{\partial}{\partial y} \mathbf{b}
\end{pmatrix}.
\]

\revOne{The manufactured solution, \eqref{eq:MMS2D} and \eqref{eq:MMS2D-bottom}, and subsequent convergence study is intended as a verification of the newly developed 2D stochastic DG spectral element discretization. Its focus on convergence in physical space is paramount. The manufactured solution is exactly represented by two stochastic modes ($K=2$).}

To investigate the convergence order of the approximation we fix the polynomial degree $N$ of the stochastic Galerkin approximation and take an increasing number of elements in each spatial direction in the domain.
For this, we use the EC flux \eqref{eq_conservative_flux_2d} in the volume of the flux differencing formulation and this EC flux with local Lax-Friedrichs-type dissipation at the interfaces of each element.
This mesh refinement study measures convergence in terms of the discrete $L_2$ error of the approximation against the manufactured solution described above.
The experimental order of convergence (EOC) is computed as
\[
\text{EOC}=\log_{10}\left(\frac{\|err_{(m)}\|_2}{\|err_{(n)}\|_2}\right)/\log_{10}\left(2\right)\;,
\]
where $m$ and $n$ denote two successive mesh resolutions.
We report the convergence results for polynomial degree $N=3$ in Table~\ref{tab:conv3} and $N=4$ in Table~\ref{tab:conv4}, respectively.
From both tables we see that the stochastic DG method is high-order convergent in space, with a convergence order of approximately $N+1$ under mesh refinement.
\begin{table}[htbp]
\centering
\caption{Convergence of discrete $L_2$ error for manufactured solution test with polynomial degree $N=3$ and $\Delta t = 5\times10^{-4}$.}
\resizebox{0.97\textwidth}{!}{%
\label{tab:conv3}
\begin{tabular}{
l
S[table-format=1.2e-2] S[table-format=1.2]
S[table-format=1.2e-2] S[table-format=1.2]
S[table-format=1.2e-2] S[table-format=1.2]
S[table-format=1.2e-2] S[table-format=1.2]
S[table-format=1.2e-2] S[table-format=1.2]
S[table-format=1.2e-2] S[table-format=1.2]
}
\toprule
& \multicolumn{2}{c}{$h_0$}
& \multicolumn{2}{c}{$h_1$}
& \multicolumn{2}{c}{$q_{10}$}
& \multicolumn{2}{c}{$q_{11}$}
& \multicolumn{2}{c}{$q_{20}$}
& \multicolumn{2}{c}{$q_{21}$} \\
\cmidrule(lr){2-3}\cmidrule(lr){4-5}\cmidrule(lr){6-7}
\cmidrule(lr){8-9}\cmidrule(lr){10-11}\cmidrule(lr){12-13}
\# elem.
& {$L_2$ error} & {EOC} & {$L_2$ error} & {EOC}
& {$L_2$ error} & {EOC} & {$L_2$ error} & {EOC}
& {$L_2$ error} & {EOC} & {$L_2$ error} & {EOC} \\
\midrule
$8^2$ & 6.58e-05 & {} & 6.87e-05 & {} & 1.06e-04 & {} & 1.05e-04 & {} & 1.28e-04 & {} & 1.27e-04 & {} \\
$16^2$ & 3.72e-06 & 4.14 & 3.79e-06 & 4.18 & 7.50e-06 & 3.82 & 7.45e-06 & 3.81 & 8.96e-06 & 3.83 & 8.89e-06 & 3.83 \\
$32^2$ & 2.62e-07 & 3.83 & 2.64e-07 & 3.84 & 5.45e-07 & 3.78 & 5.42e-07 & 3.78 & 6.31e-07 & 3.83 & 6.28e-07 & 3.82 \\
$64^2$ & 1.91e-08 & 3.78 & 1.94e-08 & 3.77 & 3.86e-08 & 3.82 & 3.84e-08 & 3.82 & 4.45e-08 & 3.83 & 4.43e-08 & 3.83 \\
\bottomrule
\end{tabular}
}
\end{table}
\begin{table}[htbp]
\centering
\caption{Convergence of discrete $L_2$ error for manufactured solution test with polynomial degree $N=4$ and $\Delta t = 5\times10^{-4}$.}
\resizebox{0.97\textwidth}{!}{%
\label{tab:conv4}
\begin{tabular}{
l
S[table-format=1.2e-2] S[table-format=1.2]
S[table-format=1.2e-2] S[table-format=1.2]
S[table-format=1.2e-2] S[table-format=1.2]
S[table-format=1.2e-2] S[table-format=1.2]
S[table-format=1.2e-2] S[table-format=1.2]
S[table-format=1.2e-2] S[table-format=1.2]
}
\toprule
& \multicolumn{2}{c}{$h_0$}
& \multicolumn{2}{c}{$h_1$}
& \multicolumn{2}{c}{$q_{10}$}
& \multicolumn{2}{c}{$q_{11}$}
& \multicolumn{2}{c}{$q_{20}$}
& \multicolumn{2}{c}{$q_{21}$} \\
\cmidrule(lr){2-3}\cmidrule(lr){4-5}\cmidrule(lr){6-7}
\cmidrule(lr){8-9}\cmidrule(lr){10-11}\cmidrule(lr){12-13}
\# elem.
& {$L_2$ error} & {EOC} & {$L_2$ error} & {EOC}
& {$L_2$ error} & {EOC} & {$L_2$ error} & {EOC}
& {$L_2$ error} & {EOC} & {$L_2$ error} & {EOC} \\
\midrule
$8^2$ & 1.17e-05 & {} & 1.15e-05 & {} & 3.10e-05 & {} & 3.10e-05 & {} & 3.68e-05 & {} & 3.68e-05 & {} \\
$16^2$ & 2.84e-07 & 5.36 & 2.83e-07 & 5.35 & 7.13e-07 & 5.44 & 7.12e-07 & 5.44 & 8.37e-07 & 5.46 & 8.36e-07 & 5.46 \\
$32^2$ & 9.00e-09 & 4.98 & 9.27e-09 & 4.93 & 1.18e-08 & 5.92 & 1.18e-08 & 5.92 & 1.37e-08 & 5.93 & 1.37e-08 & 5.94 \\
$64^2$ & 3.79e-10 & 4.57 & 3.98e-10 & 4.54 & 4.37e-10 & 4.75 & 4.40e-10 & 4.74 & 4.95e-10 & 4.79 & 4.95e-10 & 4.79 \\
\bottomrule
\end{tabular}
}
\end{table}

\subsubsection{Well-balancing}\label{sec:wb2D}

Using polynomials of degree $N=3$ in each spatial direction, $4\times 4$ Cartesian elements for the domain $[0,20]^2$, and the uncertain position bottom topography $b(x,y,\xi_1, \xi_2)$ described in \ref{sec:uncertain_bottom2D} with two, four, and eight Haar wavelets.
The initial conditions are
\[
\begin{aligned}
\mathbf{H}(x,y,t=0) &= \left\{
\begin{array}{ll}
\frac{4}{3}, & \mbox{first wavelet}\\
0, & \mbox{otherwise}
\end{array}
\right.
\qquad
\mathbf{v}_{\mathfrak{1}}(x,y,t=0) = \mathbf{v}_{\mathfrak{2}}(x,y,t=0) = 0.
\end{aligned}
\]
Again, the uncertainty in the bottom topography $b(x, y, \xi_1, \xi_2)$ propagates into the primitive variables $\mathbf{h} = \mathbf{H} - b$.
Integrating up to a final time of $t_{\text{final}} = 100$ we show the discrete $L_1$ well-balancedness error in Table~\ref{tab:2Dposition}.
\revTwo{As in the 1D well-balancedness testing, we also ran these tests using the ES flux at element interfaces and obtained similar errors.
}
\begin{table}[h!]
\centering
\caption{2D uncertain position, final time $100$, polynomial degree $3$, time step $0.1$, \revTwo{EC flux in the volume and on the surface}.}
\label{tab:2Dposition}
\begin{tabular}{
c
S[table-format=1.2e-2]
S[table-format=1.2e-2]
S[table-format=1.2e-2]
S[table-format=1.2e-2]
S[table-format=1.2e-2]
S[table-format=1.2e-2]
S[table-format=1.2e-2]
S[table-format=1.2e-2]
}
\toprule
{$K$}
& {$H_1$} & {$H_2$} & {$H_3$} & {$H_4$}
& {$H_5$} & {$H_6$} & {$H_7$} & {$H_8$} \\
\midrule
2
& 2.99374723e-15
& 1.58521704e-15
& \multicolumn{1}{c}{---}
& \multicolumn{1}{c}{---}
& \multicolumn{1}{c}{---}
& \multicolumn{1}{c}{---}
& \multicolumn{1}{c}{---}
& \multicolumn{1}{c}{---}
\\
4
& 2.37387270e-15
& 1.24888904e-15
& 9.03966372e-16
& 8.01696125e-16
& \multicolumn{1}{c}{---}
& \multicolumn{1}{c}{---}
& \multicolumn{1}{c}{---}
& \multicolumn{1}{c}{---}
\\
8
& 2.15635765e-15
& 9.40997050e-16
& 6.89736806e-16
& 6.85053221e-16
& 7.22426880e-16
& 5.28574291e-16
& 6.39784093e-16
& 7.83635571e-16
\\
\bottomrule
\end{tabular}
\end{table}

\subsubsection{Entropy conservation/stability}\label{sec:eces2D}

We first consider the domain $\Omega = [0,20]^2$.
For the flow variables we use a deterministic initial condition where the initial mean value of the water height contains a circular discontinuity.
\[
\begin{aligned}
{H}_1(x,y,t=0) &= \left\{
\begin{array}{ll}
2, & \mbox{if } \|\mathbf{x} - 10\|_2\leq 1.5\\
1.5, & \mbox{otherwise }
\end{array}
\right.
\qquad
H_{2:K}(x,y,t=0) = 0,
\qquad
\mathbf{v}_{\mathfrak{1}}(x,y, t=0)
=
\mathbf{v}_{\mathfrak{2}}(x,y, t=0)
=
0.
\end{aligned}
\]
together with a stochastic bottom topography with uncertain position described in Appendix~\ref{sec:uncertain_bottom2D}.
We run the entropy conservative (EC) test with periodic boundary conditions.
For this test, we use polynomials of degree three, $N=3$ and 64 elements in the domain.
The EC flux and nonconservative terms from \eqref{eq_conservative_flux_2d} and \eqref{eq_conservative_source_2d} are used at all interior interfaces, so that the numerical approximation is nearly dissipation-free.
We run to a final time of $t_{\text{final}} = 0.5$.
Table~\ref{tab:ec2D} reports the entropy rate at the final time for increasing number of wavelets.
\begin{table}
	\centering
	\caption{Comparison of the integrated entropy rate at the final time $t_{\text{final}} = 0.5$ for the two-dimensional dam break test with periodic boundary conditions.
    All results used $N=3$ on a $64$ element mesh.}
	\label{tab:ec2D}
	\begin{tabular}{l S[table-format=1.2e-2]}
		\toprule
		\# Haar wavelets & \multicolumn{1}{c}{$\frac{1}{|\Omega|}\int S_t \, d\Omega$}\\
		\midrule
		$K=2$ & 1.87877383e-19 \\
		$K=4$ & 3.16369280e-18 \\
		$K=8$ & 1.90700583e-18 \\
		\bottomrule
	\end{tabular}
\end{table}

Next, we run a similar configuration that is ES by adding LLF type 
dissipation to the baseline EC flux.
Here, outflow boundary conditions are considered in the $x$-direction and periodic boundary conditions are used in the $y$-direction.
We increase the resolution with $N=4$ and 64 elements and run to the same final time.
We adapt initial conditions from \cite{Dai_etal_22} that are a small perturbation of a deterministic water height and a bottom topography with uncertain position.
The deterministic initial conditions for the water height and velocities on the domain $[0,2]^2$ are
\[
\begin{aligned}
{H}_1(x,y,t=0) &= \left\{
\begin{array}{ll}
1.01, & \mbox{if  } 0.05 \leq x \leq 0.15\\
1, & \mbox{otherwise }
\end{array}
\right.
\qquad
H_{2:K}(x,y,t=0) = 0,
\qquad
\mathbf{v}_{\mathfrak{1}}(x,y, t=0)
=
\mathbf{v}_{\mathfrak{2}}(x,y, t=0)
=
0.
\end{aligned}
\]
The bottom topography with uncertain position is
\[
b(x,y, \xi_1,\xi_2) = 0.8 \, \text{exp}(-5 (x - 0.9 + 0.1 \xi_1)^2 - 50 (y - 1.0 + 0.1 \xi_2)^2)
\]
and computed with a strategy similar to that outlined in Appendix~\ref{sec:uncertain_bottom2D}.
The final time is taken to be $t_{\text{final}} = 1.8$.

Figure~\ref{fig:damBreak2D} shows the computed mean water height at four times, $t = 0.0, 0.6, 1.2,$ and $1.8$ with $K=8$ Haar wavelets.
Also present in Figure~\ref{fig:damBreak2D} are gray disks that visualize the standard deviation, $\sigma_{\text{sd}}$ of the approximation.
We also provide the evolution of the entropy over time in Figure~\ref{fig:es2D}.
We see that entropy decays module the effect of the outflow boundary conditions.
\begin{figure}[!ht]
    \centering 
\includegraphics[width=0.49\textwidth]{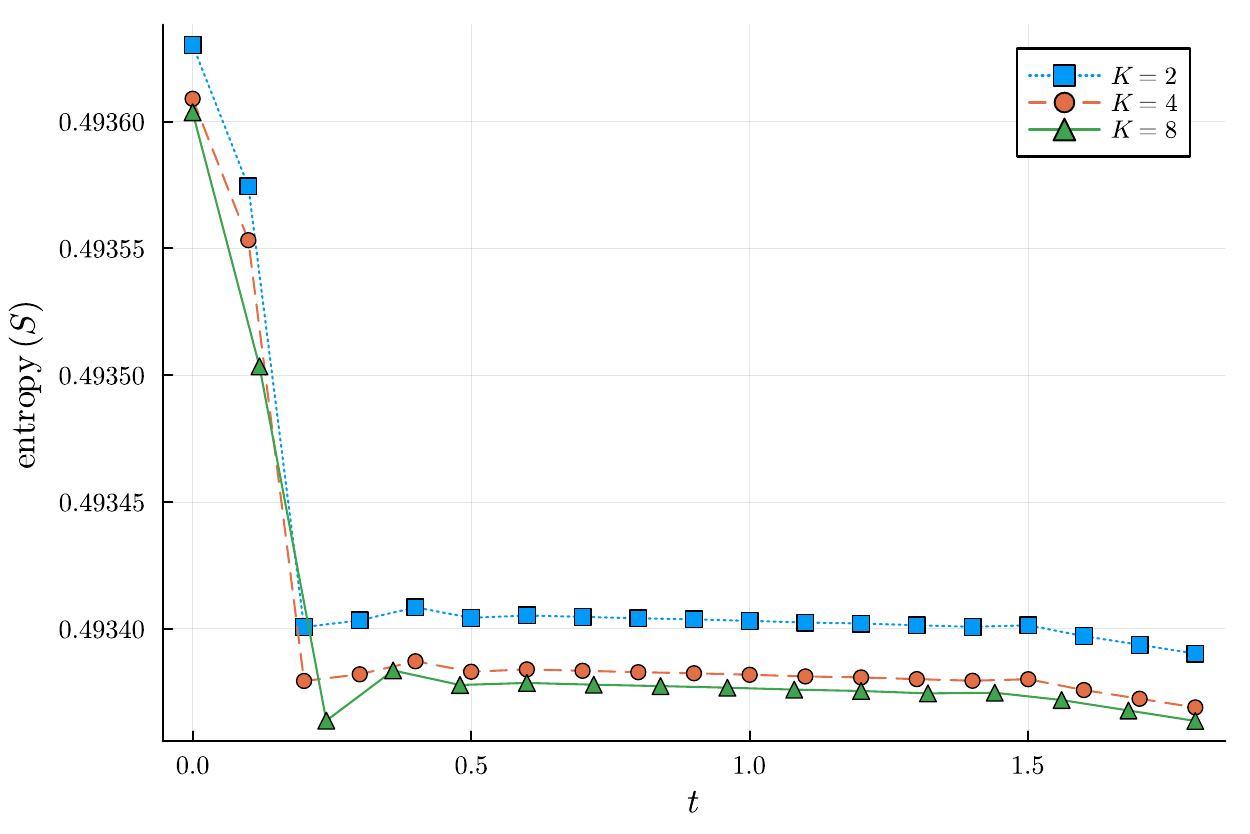}
\caption{Polynomials of degree $N=4$, $64$ elements. Evolution of the entropy for different number of wavelets using LLF type dissipation.}
\label{fig:es2D}
\end{figure}
We observe that the stochastic bottom topography induces a water height
with uncertainties particularly near the ``peak'' of the Gaussian bump bottom topography.
This is particularly evident at later times after the right traveling wave interacts with the uncertain bottom topography.
\begin{figure}[!ht]
    \centering
    \subfloat[$t=0.0$, max $\sigma_{\text{sd}}$: $0.0$]
    {
\includegraphics[width=0.495\textwidth]{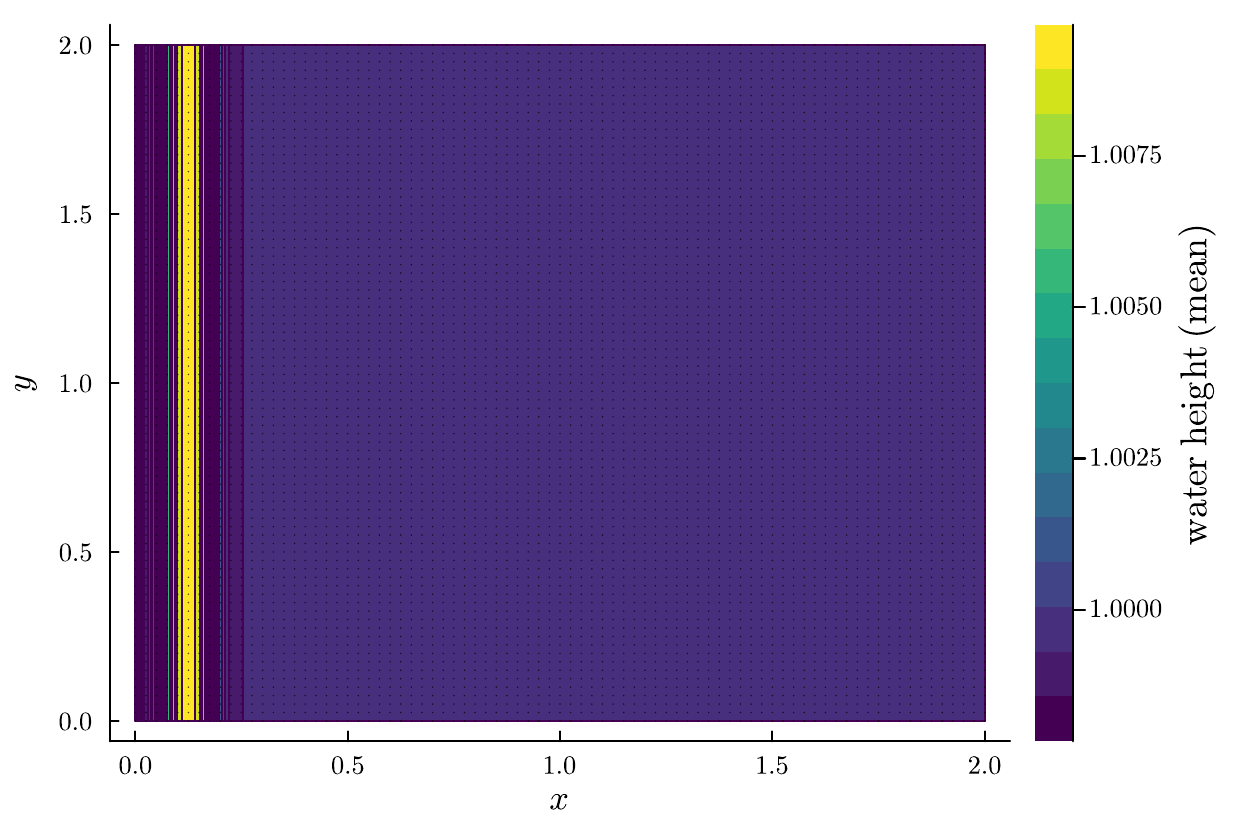}
}
    \subfloat[$t=0.6$, max $\sigma_{\text{sd}}$: $8.91\cdot 10^{-4}$]
    {
\includegraphics[width=0.495\textwidth]{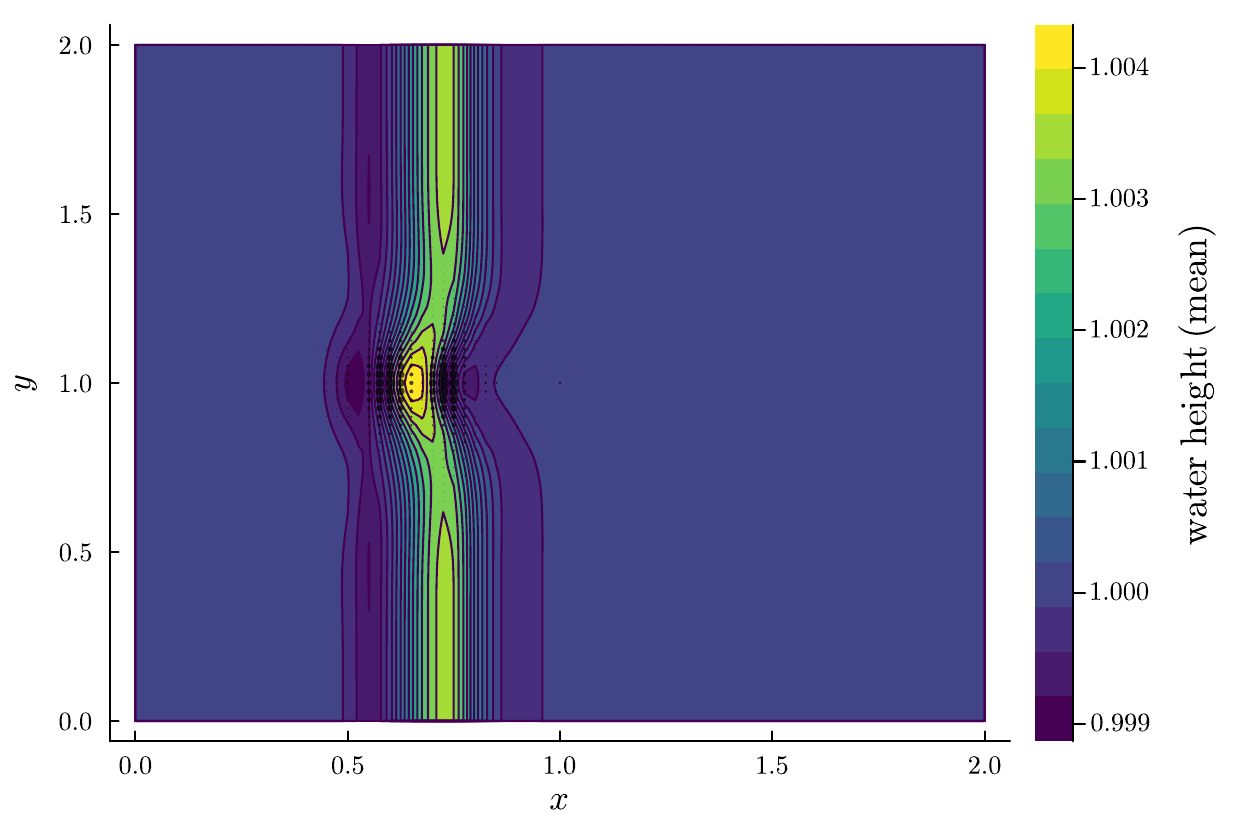}
}
\\
\subfloat[$t=1.2$, max $\sigma_{\text{sd}}$: $7.99\cdot 10^{-4}$]
    {
    \includegraphics[width=0.495\textwidth]{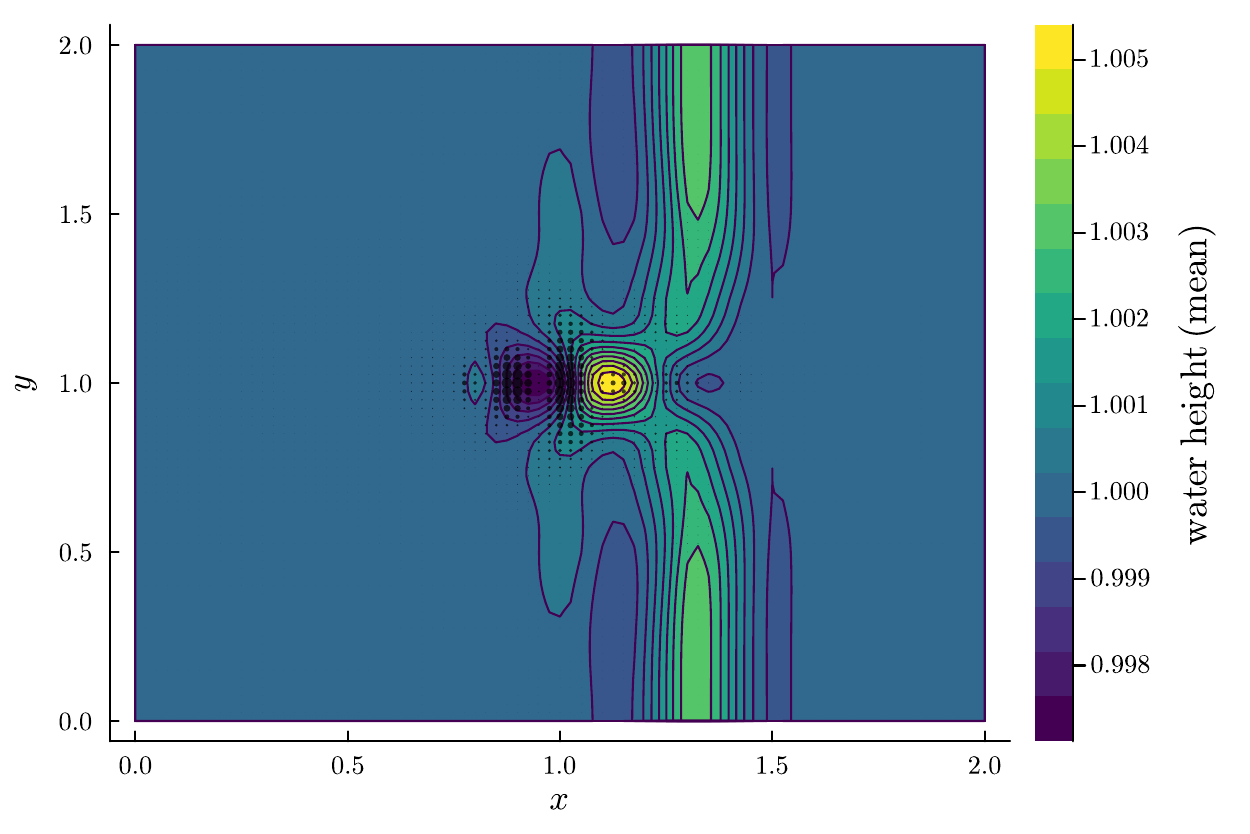}
}
\subfloat[$t=1.8$, max $\sigma_{\text{sd}}$: $2.37\cdot 10^{-4}$]
    {
    \includegraphics[width=0.495\textwidth]{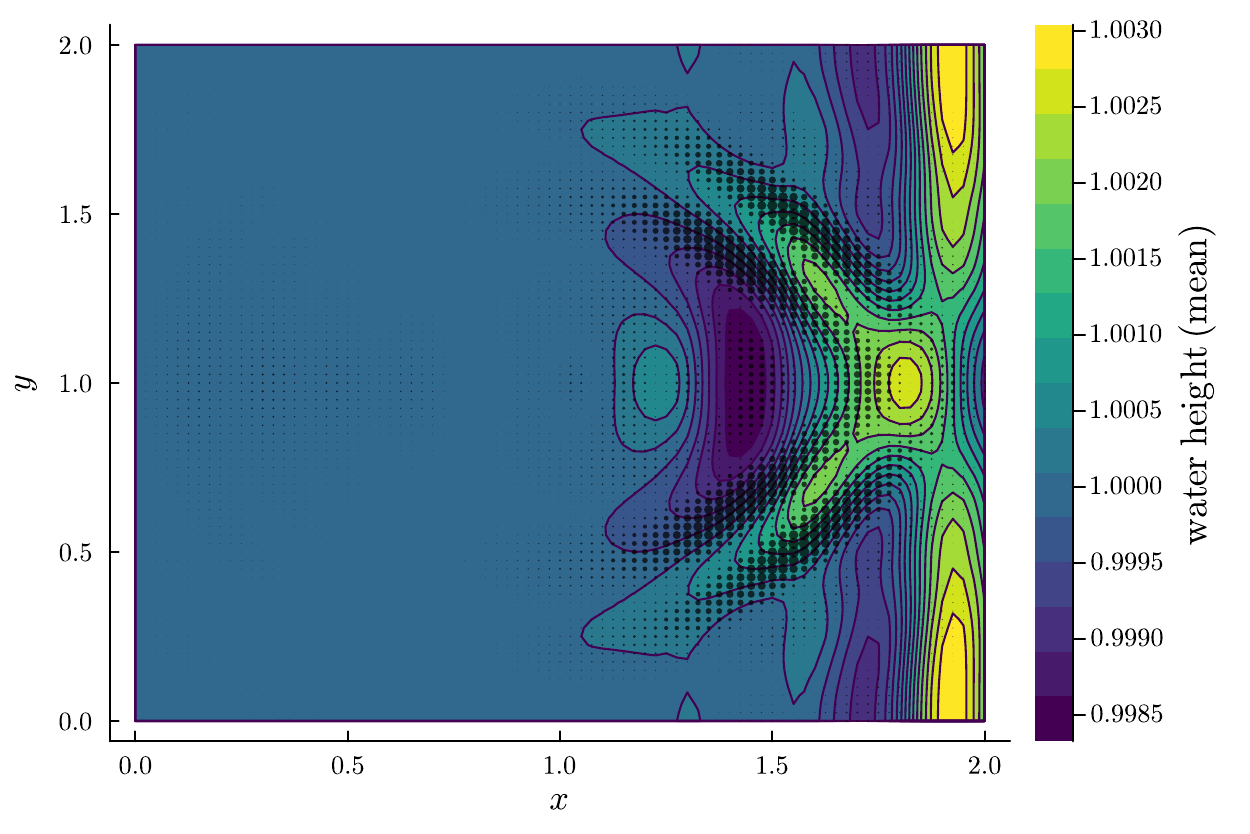}
    }
\caption{Polynomials of degree $N=4$, $64$ elements, and $8$ wavelets. Mean water surface contours with circular glyphs where the radii of the glyphs indicate the magnitude of the standard deviation.}
\label{fig:damBreak2D}
\end{figure}

%% file: Appendix.tex
\section{Appendix: details of derivations} \label{sec_Appendix}
\subsection{Proof of Theorem \ref{th_identical}}\label{sub_sub_proof_1} 
We show that all quantities are identical by elementary  calculations. 
We start with the entropy flux pair. 
In conservative variables it has been defined in \cite{dai2023energy} through \eqref{eq_ent_flux_dai}, whereas using Roe variables we have the identity \eqref{eq_ent_flux_gerster}.  The first and the last term of the entropy $\hat{\eta}_{\textbf{Roe}}$ can be shown to be identical to entries of $\hat{\eta}$ via the Roe identities
$
 \hal^T \M(\hal)^T \M(\hal) \hal = \hh^T \hh$ and $
 g\hal^T \M(\hal) \hb= g\hh^T \hb.
$
 For the second term of $\hat{\eta}_{\text{Roe}}$ in~\eqref{eq_ent_flux_gerster}, i.e., $\hbe^T \hbe$, the corresponding conservative variable formulation term is
$\hq^{T}\M^{-1}(\hh)\hq$. First, we note 
that $\hbe= \M^{-1}(\hh^{1/2})\hq$ and using the properties of the basis with \eqref{eq:basis_P1} in the last equality below, we obtain  directly 
$$
\hbe^T \hbe = 
(\M^{-1}(\hh^{1/2})\hq)^T \M^{-1}(\hh^{1/2})\hq  = \hq^T \M^{-1}(\hh^{1/2})  \M^{-1}(\hh^{1/2})\hq = \hq^{T}\M^{-1}(\hh)\hq.
$$
In terms of the entropy fluxes, the identity between the first terms $\hat{H}$ and $\hat{H}_{\text{Roe}}$ is shown by again using the property~\eqref{eq:basis_P1} 
together with the definition of $\hal$ and $\hbe$, we get
\begin{align*}
& \frac{1}{2} \hbe^T \M(\hbe)\M^{-1}(\hal) \hbe=
\frac{1}{2} \hq^T \M^{-1}(\hh^{1/2}) \M( \M^{-1} (\hh^{1/2} )\hq) \M^{-1}(\hh^{1/2}) \M^{-1}(\hh^{1/2}) \hq\\
=& 
\frac{1}{2} \hq^T \M^{}(\hh^{-1/2}) \M( \hq)  \M (\hh^{-1/2} )  \M(\hh^{-1}) \hq
=  \frac{1}{2} \hq^T \M^{}(\hh^{-1}) \M( \hq)  \M(\hh^{-1}) \hq= 
 \frac{1}{2} \hv^T\M( \hq) \hv.
\end{align*}
For the second and last terms of $H$ and $\hat{H}_{\text{Roe}}$, it follows from the definitions of the SG projections of the Roe variables that
$$
 \hal^T \M(\hal)^T \M(\hal) \hbe= \hh^T \hq \text{ and }  \hbe^T \M(\hal)^T \hb= \hq^T \hb.
 $$
This demonstrates that the entropy-flux pairs \eqref{eq_ent_flux_dai} and \eqref{eq_ent_flux_gerster} are identical. Due to the uniqueness of derivates, we obtain directly that the entropy variables are also identical and with the formula \eqref{eq_potential}
for the corresponding potential.

\subsection{Proof of Theorem \ref{th_entropy_conservative}}
\label{sub_ap_2}
\textbf{Step 1} of Subsection \ref{subsec_two_point_Dai}:\\
Represent the flux function $F$ of SG-SW in \eqref{eq:SG_SWE_cons} in entropy variables $\hat{w}$ assuming $b=0$. Using
$\hh= \frac{1}{g} \left( \hw_1+\frac12 \M(\hw_2)\hw_2 \right)$, we get 
$\hF_1(\hw)= \hq= \frac{1}{g} \left( \hw_1+\frac12 \M(\hw_2)\hw_2 \right) $
and by recalling that $ \M^2(b)a = \M(a)\M(b) b$,  we obtain 
$$\hF_2(\hw)= \frac{3}{2g}\M(\hw_1) \M(\hw_2)\hw_2 +  \frac{5}{8g} \M^3(\hw_2)\hw_2  + \frac{1}{2g}\M(\hw_1)\hw_1,$$
$$
\hat{H}(\hw)=\frac{3}{2g}\hw_2^T\M^2(\hw_2)\hw_1 + \frac{1}{2g}\hw_2^T\M^3(\hw_2)\hw_2 + \frac{1}{g}\hw_2^T\M(\hw_1)\hw_1,
$$
and 
\begin{equation}\label{en_potential}
\Phi(\hw) =  \frac{1}{2g}\hw_1^T \left[ \M(\hw_1)
+ \M^2(\hw_2) \right] \hw_2
+ \frac{1}{8g}\hw_2^T\M^3(\hw_2)\hw_2.
\end{equation}
We have to describe the jump in the potential in entropy variables. Therefore, we introduce the following identities for any $V \in \mathbb{R}^K$: 
\begin{align*}
\jump{\M^3(V)V} =& \M(\avg{V})\jump{\M^2(V)V} + \M(\jump{V})\avg{\M^2(V)V}\\
=& \M(\avg{V}) \left( \M(\jump{V})\avg{\M(V)V} + \M(\avg{V})\jump{\M(V)V} \right) + \M(\jump{V})\avg{\M^2(V)V}\\
=& \M(\avg{V})\M(\jump{V})\avg{\M(V)V} + \M^{2}(\avg{V})\M(\jump{V})\avg{V} + \M^{3}(\avg{V})\jump{V} + \M(\jump{V})\avg{\M^2(V)V}\\
=& \M(\jump{V})\M(\avg{V})\avg{\M(V)V}  + 2\M^{3}(\avg{V})\jump{V} + \M(\jump{V})\avg{\M^2(V)V},
\end{align*}
$$
\jump{\M^2(V)V}= \jump{\M(V)} \avg{\M(V) V} + \avg{\M(V)}\jump{\M(V) V}= \jump{\M(V)} \avg{\M(V) V} + \avg{\M(V)}\jump{\M(V)} \avg{V} + \avg{\M(V)} \cdot \avg{\M(V)} \jump{ V}, % \jump{\M(w)} \avg{\M (w)w} +\avg{\M^2(w)} \jump{w}
$$
and
\begin{align*}
\jump{V^T \M^3(V)V} =& \jump{V^T}\avg{\M^3(V)V} + \avg{V}^T \jump{\M^3(V)V}\\
=& \jump{V^T}\avg{\M^3(V)V} + \jump{V^T} \M^2(\avg{V})\avg{\M(V)V}
+ 2 \jump{V^T} \M^3(\avg{V})\avg{V} + \jump{V^T}\M(\avg{V})\avg{\M^2(V)V}.
\end{align*}
Then, the jump in the entropy potential~\eqref{en_potential} can be written
\begin{align}
\jump{\Phi} = \frac{1}{2g} \left( 
\jump{\hw_1^T}\avg{\M(\hw_{1})\hw_2} + \avg{\hw_1^T}\M(\jump{\hw_1})\avg{\hw_2} + \avg{\hw_1^T}\M(\avg{\hw_1})\jump{\hw_2}
+\jump{\hw_1^T}\avg{\M^2(\hw_2)\hw_2} \right) \\
+ \frac{1}{2g}\left(
 \avg{\hw_1^T}\M(\jump{\hw_2})\avg{\M(\hw_2)\hw_2} + 2 \avg{\hw_1^T}\M^2(\avg{\hw_2})\jump{\hw_2}
\right)\\
+\frac{1}{8g}\jump{\hw_2^T}\left( 
\avg{\M^3(\hw_2)\hw_2} + \M^2(\avg{\hw_2}) \avg{\M(\hw_2)\hw_2} + 2\M^3(\avg{\hw_2})\avg{\hw_2} + \M(\avg{\hw_2}) \avg{\M^2(\hw_2)\hw_2}
\right).
\end{align}
Identifying terms that have either factors $\jump{\hw_1^T}$ or $\jump{\hw_2^T}$, we split $\jump{\Phi}$ according to:
\begin{equation}
\begin{aligned}
\jump{\Phi_1} &= \frac{1}{2g}\jump{\hw_1^T} 
\left(\avg{\M(\hw_1)\hw_2} + \M(\avg{\hw_1}) \avg{\hw_2}
+ \avg{\M^2(\hw_2)\hw_2}\right),\\[0.1cm]
\jump{\Phi_2} &= 
\frac{1}{2g}\jump{\hw_2^T} 
\left( \M(\avg{\hw_1})\avg{\hw_1} +\M(\avg{\hw_1})\avg{\M(\hw_2)\hw_2} + 2 \M^2(\avg{\hw_2})\avg{\hw_1}\right. \\[0.1cm]
&\quad\left.
+\frac{1}{4}\left( 
\avg{\M^3(\hw_2)\hw_2} + \M^2(\avg{\hw_2})\avg{\M(\hw_2)\hw_2} + 2\M^3(\avg{\hw_2})\avg{\hw_2} + \M(\avg{\hw_2})\avg{\M^2(\hw_2)\hw_2} \right) 
 \right).
\end{aligned}
\end{equation}
Hence, by using \eqref{eq_Tadmor}, we obtain an entropy conservative numerical flux by comparison. It is given by  
\begin{equation*}
\begin{aligned}
\hF^{\#}_{1}(\hw) &= \frac{1}{2g} 
\left(\avg{\M(\hw_1)\hw_2} + \M(\avg{\hw_1}) \avg{\hw_2}
+ \avg{\M^2(\hw_2)\hw_2}\right),\\[0.1cm]
F^{\#}_{2}(\hw) &= \frac{1}{2g} 
\left( \M(\avg{\hw_1})\avg{\hw_1} +\M(\avg{\hw_1})\avg{\M(\hw_2)\hw_2} + 2 \M^2(\avg{\hw_2})\avg{\hw_1}\right. \\[0.1cm]
&\quad\left.
+\frac{1}{4}\left( 
\avg{\M^3(\hw_2)\hw_2} + \M^2(\avg{\hw_2})\avg{\M(\hw_2)\hw_2} + 2\M^3(\avg{\hw_2})\avg{\hw_2} + \M(\avg{\hw_2})\avg{\M^2(\hw_2)\hw_2} \right) 
 \right).
\end{aligned}
\end{equation*}
\textbf{Step 2:}
Using the identities~\eqref{eq:entropy_variable_Dai}
with $\hat{b}=0$, we obtain the following expressions :
\begin{equation*}
\begin{aligned}
F^{\#}_{1}(\hu) &= \frac{1}{2g} \left(
\frac{1}{2}\avg{\M^2(\hv)\hv} + g\avg{\M(\hv)\hh} -\frac{1}{2}\avg{\M^2(\hv)} \ \avg{\hv} + g\M(\avg{\hv}) \avg{\hh}
\right), \\
F^{\#}_{2}(\hu) &= \frac{1}{2g} \left(
\frac{1}{4} \avg{\M^2(\hv)}\ \avg{\M(\hv)\hv} + g^2 \M(\avg{\hh})\avg{\hh} - g\avg{\M^2(\hv)}\ \avg{\hh}
-\frac{1}{2}\avg{\M^2(\hv)} \ \avg{\M(\hv)\hv} + g\M(\avg{\hh})\ \avg{\M(\hv)\hv} \right. \\
& \quad \left. 
-\M^2(\avg{\hv})\ \avg{\M(\hv)\hv}
+ 2g\M^2(\avg{\hv})\ \avg{\hh}
+\frac{1}{4}\left( \avg{\M^3(\hv)\hv} + \M^2(\avg{\hv})\ \avg{\M(\hv)\hv}
+ 2\M^3(\avg{\hv}) \hv + \M(\avg{\hv})\ \avg{\M^2(\hv)\hv} \right)
\right),
\end{aligned}
\end{equation*}
which demonstrates \eqref{eq_flux_conservative}.

\textbf{Step 3:} In the last step, we incorporate the source term and the fact that we have non-zero bottom topography with the entropy variables. Including bottom topography alters the entropy variable $\hw_1$ to be
\[
\hw_1= g\left(\hh+ \hb\right)+ \frac{1}{2g}\M(\hv)\hv.
\]
Using the identity \eqref{eq_Tadmor_2}, we obtain that the source term has to be discretized so that \eqref{eq_calculated} is fulfilled, i.e.,

\begin{equation}
g F^{\#}_{1,i,k} \cdot \jump{\hat{b}}_{i,k}+ \hv^T_k \cdot S_{k,i}^\#- \hv^T_i \cdot S_{i,k}^\#=0,
\end{equation}
where $F^{\#}_{1,i,k}$ is the first component of the numerical flux calculated between the $i^\text{th}$ and $k^\text{th}$ degree of freedom. Since the first part of the numerical fluxes \eqref{eq_flux_conservative} and \eqref{eq_conservative_flux_dai} have the same contribution, the source
can be handled analogously. We split the source $S_{i,k}^\#= \frac{1}{2} g \left(\M (\overline{\hh}_{i,k}) \jump{\hat b}_{i,k} \right) + \tilde{S}_2^\#$ therefore. 
Then, the extended numerical flux is given by $F^{\#,ext}_{2,i,k}= F_{2,i,k}^\#+ \frac{1}{2g} \M (\overline{\hh}_{i,k}) \jump{\hat b}_{i,k} +\tilde{S}_{2,i,k}^\#$
A second contribution denoted by $\tilde{S}_{2,i,k}$ must be included such that the following equation holds

\begin{equation}\label{eq_basic1}
\begin{aligned}
 \frac{1}{2g} \left( \left( (\avg{\M^2(\hv) \hv})_{i,k}-  \M(\avg{\M(\hv) (\hv)})_{i,k} \avg{ (\hv)_{i,k}} \right)\right)^T \jump{\hat{b}}_{i,k}+
\hv^T_k \tilde{S}_{k,i}^\#- \hv^T_i \tilde{S}_{i,k}^\#
=0.
\end{aligned}
\end{equation}
This additional term comes from   the first numerical flux component $F^\#_1$. We obtain 
\begin{equation*}
\begin{aligned}
&\left( (\avg{\M^2(\hv) \hv})_{i,k}-  \M(\avg{\M(\hv) (\hv)})_{i,k} \avg{ (\hv)_{i,k}} \right)
=\frac{(\M^2(\hv) \hv)_i +(\M^2(\hv) \hv)_k }{2}\\
&- \frac{(\M(\M(\hv) (\hv))_i +\M(\M(\hv) (\hv))_k)(u_i+u_k)}{4}\\
&=\frac{(\M^2(\hv) \hv)_i +(\M^2(\hv) \hv)_k - (\M^2 (\hv))_i u_k -  (\M^2 (\hv))_k u_i   }{4} = \frac{\M^2(\hv)_i (\hv_i-\hv_k) -(\M^2(\hv) )_k (\hv_i- \hv_k)  }{4} \\
&= \frac{1}{4}  \left((\M^2(\hv) )_i- (\M^2(\hv) )_k  \right) \left( (\hv_i-\hv_k) \right)= 
 \frac{1}{4}  \left((\M(\hv) )_i- (\M(\hv) )_k  \right) \left((\M(\hv) )_i+(\M(\hv) )_k  \right) \left( \hv_i-\hv_k\right). 
\end{aligned}
\end{equation*}
Thus, choosing $\tilde{S}_{i,k}^\#= \frac{1}{8g}  \left((\M(\hv) )_k- (\M(\hv) )_i  \right)^2 \jump{\hat{b}}_{i,k}$ results in the desired equality \eqref{eq_basic1} since we have
\begin{equation*}\label{eq_basic_2}
\begin{aligned}
& \frac{1}{8g} \left( \left((\M(\hv) )_i- (\M(\hv) )_k  \right) \left((\M(\hv) )_i+(\M(\hv) )_k  \right) \left( \hv_i-\hv_k\right) \right)^T \jump{\hat{b}}_{i,k}\\
& +\frac{1}{8g} 
\hv^T_k  \left((\M(\hv) )_i- (\M(\hv) )_k  \right)^2 \jump{\hat{b}}_{k,i}- \frac{1}{8g} \hv^T_i \left((\M(\hv) )_k- (\M(\hv) )_i  \right)^2 \jump{\hat{b}}_{i,k}\\
&= \frac{1}{8g} \Bigg( \left( \left((\M(\hv) )_i- (\M(\hv) )_k  \right) \left((\M(\hv) )_i+(\M(\hv) )_k  \right) \left( \hv_i-\hv_k\right) \right)^T \\
&\qquad- \hv^T_k  \left((\M(\hv) )_i- (\M(\hv) )_k  \right)^2 - \hv^T_i \left((\M(\hv) )_k- (\M(\hv) )_i  \right)^2  \Bigg) \jump{\hat{b}}_{i,k}=0.
\end{aligned}
\end{equation*}
Therefore, we have\footnote{Here, we correct a typo from \cite{bender_2023} where the second part is missing. }
\begin{equation*}\label{eq:bterm}
S^{\#}_{i,k} = \begin{pmatrix}
0 \\
\frac{1}{2} g \left(\M (\overline{\hh}_{i,k}) \jump{\hat b}_{i,k} \right)+  \frac{1}{8g}  \left((\M(\hv) )_k- (\M(\hv) )_i  \right)^2 \jump{\hat{b}}_{i,k}
\end{pmatrix}.
\end{equation*}
Thus, we obtain the complete source term discretization for the second component of entropy conservative flux. This completes the proof of Theorem \ref{th_entropy_conservative}.

\section{Expressions for initial conditions}
%{\color{blue}This section will be removed after implementation of the initial conditions. Some expressions may be recycled in the results section.}
\subsection{Uncertain bottom topography, 1D}\label{sec:uncertain_bottom1D}
First consider a bottom topography where the \emph{location} of the bump function is uncertain:
\[
b(x, \xi) =
\left\{
\begin{array}{ll}
1-\frac{1}{4}(x-x_0(\xi))^2 & \mbox{if } |x-x_0| \leq 2\\
0 & \mbox{otherwise}
\end{array}
\right.
\]
with $x_0 = 10+c\xi$, $c>0$. Let the end points and mid point of the domain of wavelet $\psi_k$ be denoted $d_k^{L}$, $d_k^{U}$, and $d_k^{M}$, respectively. Then the projection can be expressed as the sum of two smooth integrals,
\[
b_k(x) = \frac{1}{2}\int_{\max(d_k^{L},\ \min(\frac{x-12}{c}),\ d_k^M )}^{\min(d_k^{M},\ \max(\frac{x-8}{c}),\ d_k^L )}
\left( 1-\frac{1}{4}(x-x_0(\xi))^2 \right)\psi_k(\xi)\textup{d}\xi
+
\frac{1}{2}\int_{\max(d_k^{M},\ \min(\frac{x-12}{c}),\ d_k^U )}^{\min(d_k^{U},\ \max(\frac{x-8}{c}),\ d_k^M )}
\left( 1-\frac{1}{4}(x-x_0(\xi))^2 \right)\psi_k(\xi)\textup{d}\xi
\]
This expression is straightforward to evaluate with numerical quadrature pointwise in $x$, but has no easy closed-form expression. Next, we consider a random scaling with unit expectation of the same bump function with no uncertainty in its location:
\[
b(x, \xi) =
\left\{
\begin{array}{ll}
(1+c\xi)\left( 1-\frac{1}{4}(x-x_0)^2 \right) & \mbox{if } |x-x_0| \leq 2\\
0 & \mbox{otherwise}
\end{array}
\right.
\]
with $0<c<1$ and $x_0=10$. In this case, the projections onto the Haar wavelets are much easier to evaluate:
\[
b_1 = 
\left\{
\begin{array}{ll}
1-\frac{1}{4}(x-x_0)^2 & \mbox{if } |x-x_0| \leq 2\\
0 & \mbox{otherwise}
\end{array}
\right.
\]
and for $k>1$:
\[
b_k = 
\left\{
\begin{array}{ll}
-\frac{c 2^{-1.5\ell} }{2}(1-\frac{1}{4}(x-x_0)^2) & \mbox{if } |x-x_0| \leq 2\\
0 & \mbox{otherwise}
\end{array}
\right.
\]
where $\ell$ denotes the level of the wavelets, starting from 0. In practice, the deterministic bump function is simply to be scaled by each of the entries of $[1, -c/2, -c2^{-5/2}, -c2^{-5/2}, -c2^{-4}, -c2^{-4}, -c2^{-4}, -c2^{-4}]$ for the case of 8 basis functions.

\begin{figure}
\centering 
\subfloat{\includegraphics[width=0.4\textwidth]{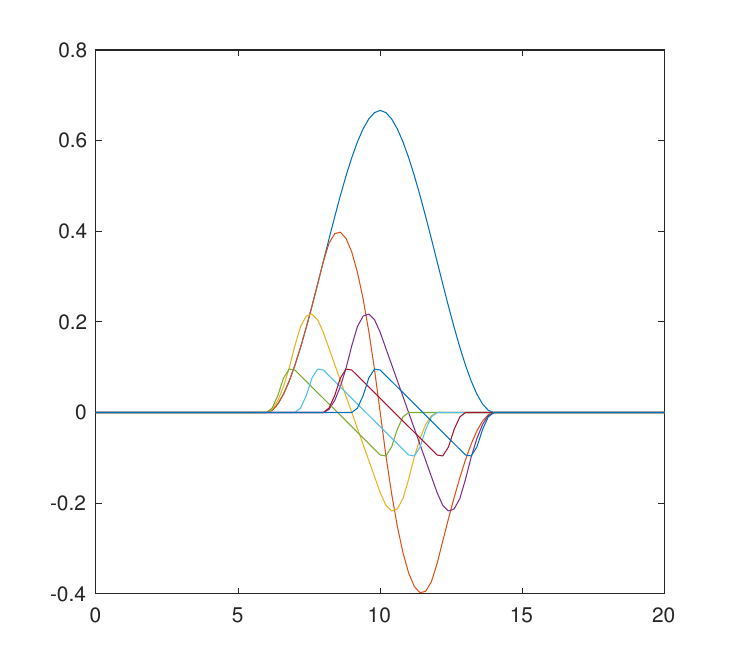}}
\subfloat{\includegraphics[width=0.4\textwidth]{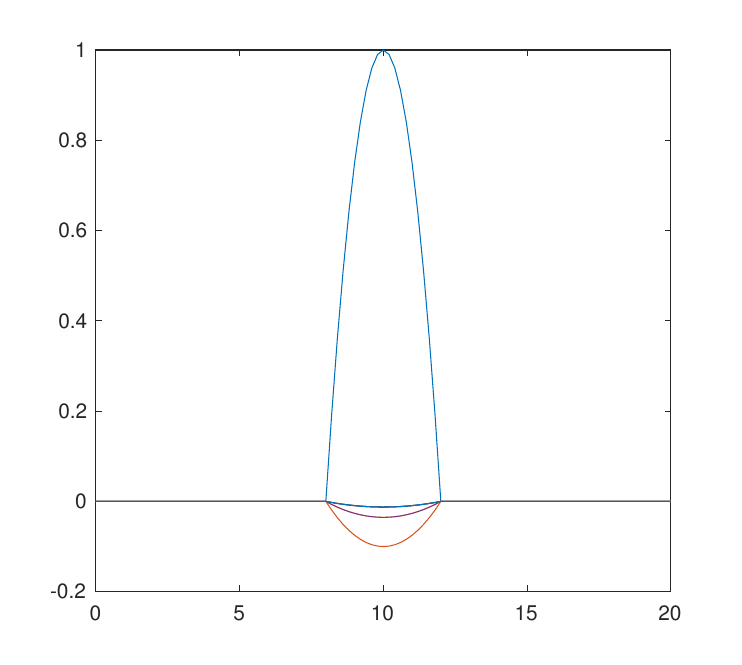}}
\caption{First eight Haar wavelets coefficients for uncertain bump location (left), and uncertain bump height (right). Note that some coefficient functions are identical in the latter case.}
\label{fig}
\end{figure}

\subsection{Uncertain bottom topography, 2D}\label{sec:uncertain_bottom2D}
Consider the 2D generalization of the bottom topography where the location of the bump is uncertain:
\[
b(x, \xi) =
\left\{
\begin{array}{ll}
1-\frac{1}{4^2}(x-x_0(\xi_1))^2(y-y_0(\xi_2))^2 & \mbox{if } |x-x_0(\xi_1)| \leq 2 \mbox{ and } |y-y_0(\xi_2)| \leq 2 \\
0 & \mbox{otherwise}
\end{array}
\right.
\]
%{\color{blue}[The expression below seems correct (tested in Matlab). The notation will be modified and explained. Note that the 2D bump is not radial, so it has a slightly square appearance.]}
Let the support of $\psi_k$ be $[d_k^{\text{L},j} d_k^{\text{U},j}]$ with mid-point $d_k^{\text{M},j}$ in the $j$th random variable, $j=1,2$. 
With $x_0(\xi_1)=10+c \xi_1$ and $y_0(\xi_2) = y + c \xi_2$, we get the following expression suitable for numerical quadrature since each integral is smooth:
\begin{multline}
b_k(x) = \\
%% Int 1
= \frac{1}{4} 
\int_{\max\left(d_k^{L,2},\ \min\left[\frac{y-12}{c},\ d_k^{M,2} \right]\right)}^{\min \left(d_k^{M,2},\ \max\left[\frac{y-8}{c},\ d_k^{L,2} \right] \right)}
\int_{\max \left(d_k^{L,1},\ \min \left[ \frac{x-12}{c},\ d_k^{M,1} \right] \right) }^{\min \left(d_k^{M,1},\ \max \left[ \frac{x-8}{c},\ d_k^{L,1} \right] \right)}
\left( 1-\frac{1}{4^2}(x-x_0(\xi_1))^2(y-y_0(\xi_2))^2 \right)\psi_k(\xi_1,\xi_2)\textup{d}\xi_1 \textup{d}\xi_2 \\
%% Int 2
+
\frac{1}{4} 
\int_{\max \left(d_k^{M,2},\ \min \left[\frac{y-12}{c},\ d_k^{U,2} \right] \right)}^{\min \left(d_k^{U,2},\ \max \left[\frac{y-8}{c},\ d_k^{M,2} \right) \right] }
\int_{\max \left(d_k^{L,1},\ \min \left[\frac{x-12}{c},\ d_k^{M,1} \right] \right)}^{\min \left(d_k^{M,1},\ \max \left[\frac{x-8}{c},\ d_k^{L,1} \right] \right)}
\left( 1-\frac{1}{4^2}(x-x_0(\xi_1))^2(y-y_0(\xi_2))^2 \right)\psi_k(\xi_1,\xi_2)\textup{d}\xi_1 \textup{d}\xi_2 \\
%%% Int 3
+
\frac{1}{4} 
\int_{\max \left(d_k^{L,2},\ \min \left[ \frac{y-12}{c},\ d_k^{M,2} \right] \right)}^{\min \left( d_k^{M,2},\ \max \left[\frac{y-8}{c},\ d_k^{L,2} \right] \right)}
\int_{\max\left( d_k^{M,1},\ \min \left[ \frac{x-12}{c},\ d_k^{U,1} \right] \right)}^{\min \left(d_k^{U,1},\ \max \left[ \frac{x-8}{c},\ d_k^{M,1} \right] \right)}
\left( 1-\frac{1}{4^2}(x-x_0(\xi_1))^2(y-y_0(\xi_2))^2 \right)\psi_k(\xi_1,\xi_2)\textup{d}\xi_1 \textup{d}\xi_2 \\
+
\frac{1}{4} 
\int_{\max \left(d_k^{M,2},\ \min \left[\frac{y-12}{c},\ d_k^{U,2} \right] \right)}^{\min \left(d_k^{U,2},\ \max \left[\frac{y-8}{c},\ d_k^{M,2} \right] \right)}
\int_{\max \left( d_k^{M,1},\ \min \left[\frac{x-12}{c},\ d_k^{U,1} \right] \right)}^{\min \left(d_k^{U,1},\ \max \left[\frac{x-8}{c},\ d_k^{M,1} \right] \right)}
\left( 1-\frac{1}{4^2}(x-x_0(\xi_1))^2(y-y_0(\xi_2))^2 \right)\psi_k(\xi_1,\xi_2)\textup{d}\xi_1 \textup{d}\xi_2 
\end{multline}
Note that the integration limits are spatially dependent, so the integrals are computed for each numerical grid point.

%{\color{blue}{TODO: Fix other brackets in limits. Mention that limits are spatially varying. Also, if function changes with uncertain position the limits must be adjusted accordingly.}}